\documentclass[11pt,leqno]{amsart}%
\usepackage{amsmath}
\usepackage{amsfonts,mathrsfs}
\usepackage{amssymb}
\usepackage{graphicx}
\usepackage{amsthm, amscd,indentfirst}
\usepackage{esint}
\usepackage{comment}      

\usepackage{graphicx}
\usepackage{epstopdf}
\usepackage{latexsym}
\usepackage{color}

\usepackage{hyperref}
\usepackage[margin=1.90cm]{geometry}
\providecommand{\U}[1]{\protect\rule{.1in}{.1in}}
\newtheorem{theorem}{Theorem}[section]
\newtheorem*{acknowledgement*}{Acknowledgement}

\newtheorem{corollary}[theorem]{Corollary}

\newtheorem{lemma}[theorem]{Lemma}

\newtheorem{problem}[theorem]{Problem}

\newtheorem{remark}[theorem]{Remark}

\newcommand{\diver}{\mathrm{div}}

\newcommand{\disp}{\displaystyle}

\newcommand{\Ricc}{\mathrm{Ric}}
\newcommand{\loc}{\mathrm{loc}}
\newcommand{\eps}{\varepsilon}
\newcommand{\R}{\mathbb{R}}
\newcommand{\Sec}{\mathrm{Sec}}
\newcommand{\HH}{\mathbb{H}}
\newcommand{\vol}{\mathrm{vol}}
\newcommand{\PP}{\mathscr{P}}
\newcommand{\di}{\mathrm{d}}
\newcommand{\dou}{\mathscr{D}}
\newcommand{\Po}{\mathscr{P}}
\newcommand{\lip}{\mathrm{Lip}}
\newcommand{\RCD}{\mathsf{RCD}}

\newcommand{\me}{\mathsf{m}}

\newcommand{\NN}{\mathbb{N}}
\newcommand{\vm}{\mathbb{V}}

\newcommand{\CCW}{C_{c}^{\infty}\hookrightarrow W^{k,p}}

\makeatletter
\@namedef{subjclassname@2020}{%
  \textup{2020} Mathematics Subject Classification}
\makeatother

\title[Density and non-density of $\protect\CCW$]{Density and non-density of $C^\infty_c \hookrightarrow W^{k,p}$ on complete  manifolds with curvature bounds}

\author[Shouhei Honda]{Shouhei Honda}
\address[Shouhei Honda]{Mathematical Institute, Tohoku University, Aoba Aramaki, Aoba Ward, Sendai, Miyagi 980-0845, Japan}
\email{shouhei.honda.e4@tohoku.ac.jp}

\author[Luciano Mari]{Luciano Mari}
\address[Luciano Mari]{Dipartimento di Matematica "Giuseppe Peano", Universit\`a degli Studi di Torino, Via Carlo Alberto 10, I-10123 Torino, Italy}
\email{luciano.mari@unito.it}

\author[Michele Rimoldi]{Michele Rimoldi}
\address[Michele Rimoldi]{Dipartimento di Scienze Matematiche ''Giuseppe Luigi Lagrange", Politecnico di Torino, Corso Duca degli Abruzzi, 24, I-10129 Torino, Italy}
\email{michele.rimoldi@polito.it}

\author {Giona Veronelli}
\address[Giona Veronelli]{Dipartimento di Matematica e Applicazioni, Universit\`a di Milano-Bicocca, via R. Cozzi 53, I-20125 Milano, Italy}
\email{giona.veronelli@unimib.it}

\begin{document}
\begin{abstract}
We investigate the density of compactly supported smooth functions in the Sobolev space $W^{k,p}$ on complete Riemannian manifolds. In the first part of the paper, we extend to the full range $p\in [1,2]$ the most general results known in the Hilbertian case. In particular, we obtain the density under a quadratic Ricci lower bound (when $k=2$) or a suitably controlled growth of the derivatives of the Riemann curvature tensor only up to order $k-3$ (when $k>2$). To this end, we prove a gradient regularity lemma that might be of independent interest. In the second part of the paper, for every $n \ge 2$ and  $p>2$ we construct a complete $n$-dimensional manifold with sectional curvature bounded from below by a negative constant, for which the density property in $W^{k,p}$ does not hold for any $k \ge 2$. We also deduce the existence of a counterexample to the validity of the Calder\'on-Zygmund inequality for $p>2$ when $\Sec \ge 0$, and in the compact setting we show the impossibility to build a Calder\'on-Zygmund theory for $p>2$ with constants only depending on a bound on the diameter and a lower bound on the sectional curvature. 
\end{abstract}


\maketitle

\tableofcontents

\section{Introduction}
In the last decades there was a lot of effort put into a better understanding of Sobolev spaces on non-compact Riemannian manifolds. On the one hand, in the Euclidean spaces one has different equivalent definitions of Sobolev spaces. Once these definitions are transposed on a Riemannian manifold, one would like to know if they remain  equivalent or not (see the introduction of \cite{V-counter} for a brief survey on this topic). On the other hand, it is useful to know which of the nice properties enjoyed by Sobolev spaces on $\R^n$ still hold in the setting of non-compact manifolds. 

Consider a complete, $n$-dimensional Riemannian manifold without boundary $(M, g)$. Let $W^{k,p}(M)$ be the Sobolev space of functions on $M$ all of whose covariant derivatives of order $j$ (in the distributional sense) are tensor fields with finite $L^{p}$-norm, for $0\leq j\leq k$. This turns out to be a Banach space, once endowed with the natural norm
\[
\left\|u\right\|_{W^{k,p}(M)}\dot =\sum_{j=0}^{k}\left(\int_{M}|\nabla^{j}u|^{p}\right)^{\frac{1}{p}}.
\]

By a generalised Meyers-Serrin-type theorem (see e.g. \cite{GGP-AccFinn}),  the set $C^\infty(M)\cap W^{k,p}(M)$ is dense in $W^{k,p}(M)$. This actually holds without assuming completeness of $M$. However, it is not a-priori obvious whether the smaller subset $C^\infty_c(M)$ of compactly supported functions is still dense. Having defined the space $W_{0}^{k,p}(M)\subseteq W^{k,p}(M)$ as the closure of $C^\infty_c(M)$ with respect to the norm $\left\|\cdot\right\|_{W^{k,p}(M)}$, our paper gives a contribution to the following problem:

\begin{problem}\label{Pb}
Let $k\geq 0$ be an integer and let $p\in[1,\infty)$. Under which assumptions on $(M,g)$, $k$ and $p$ is it true that
\begin{equation}\label{eq_dens}
W_{0}^{k,p}(M)=W^{k,p}(M)\,?
\end{equation} 
\end{problem}

\vspace{0.2cm}
\noindent\textbf{Notation.} Hereafter, we fix a function $\lambda:[0,\infty]\to (0,\infty)$ for which there exists a constant $K\in\ \mathbb{N}$ such that
\begin{equation}\label{lambda}
	\lambda(t)\dot=t\prod_{j=1}^{K}\ln^{[j]}(t),\qquad \text{for }t\gg1,
\end{equation}
where $\ln^{[j]}$ stands for the $j$-th iterated logarithm (e.g. $\ln^{[2]}(t)=\ln\ln t$, etc.).
Hereafter, all manifolds considered will have no boundary. Moreover, given a Riemannian manifold $(M,g)$, we denote with $r(x)$ the Riemannian distance from a fixed origin $o\in M$ and by $B_R(x)$ the geodesic ball of radius $R$ centered at a point $x\in M$. Also, given real-valued functions $f_1$ and $f_2$, we write $f_1\lesssim f_2$ to mean that there exists a constant $C>0$ such that $f_1 \le C f_2$. With an abuse of notation, we agree that given a tensor $T$ the symbol $\left\|T\right\|_{L^{p}(\Omega)}$ will denote the $L^{p}$ norm of the function $|T|$ on $\Omega$; for instance, we will write $\left\|\nabla f\right\|_{L^{p}(\Omega)}$ instead of $\left\||\nabla f|\right\|_{L^{p}(\Omega)}$.

\vspace{0.2cm}

Problem \ref{Pb} has a long history. It is a standard fact that, without assuming completeness, $W_{0}^{0,p}(M)=W^{0,p}(M)=L^p(M)$, and with a little effort one can also prove that $W_{0}^{1,p}(M)=W^{1,p}(M)$ for all $p\in[1,\infty)$ on any complete manifold, \cite{aubin-bull}. Also, it is obvious that $W_{0}^{k,p}(M)=W^{k,p}(M)$ for all $k\geq 0$ and $p\in[1,\infty)$ whenever $M$ is compact (see for instance \cite{HebeyCourant}).
Concerning the non-trivial case $k\geq 2$, several partial positive results have been proved: a non-exhaustive list of contributions include works by T. Aubin \cite{aubin-bull}, J. Eichhorn \cite{Eichhorn2,Eichhorn1}, E. Hebey \cite{Hebey,HebeyCourant}, L. Bandara \cite{Bandara}, B. G\"uneysu \cite{Guneysu-Book}, B. G\"uneysu and S. Pigola \cite{GuneysuPigola}, and D. Impera, M. Rimoldi, and G. Veronelli \cite{IRV-HessCutOff,IRV-HO}. To the best of our knowledge, the most general and up-to-date result is the following theorem from \cite{IRV-HO}, which generalizes previously known achievements and goes far beyond the case of constant bounds on the curvature and the specific second order case ($k=2$).

\begin{theorem}[see Theorem 1.5  and Theorem 1.7 in \cite{IRV-HO}]\label{t: IRV}
	Let $(M,g)$ be a complete Riemannian manifold, and define $\lambda$ as in \eqref{lambda}. Then, 
\begin{itemize}
	\item[(i)] 
  $W^{k,p}(M)=W^{k,p}_0(M)$ for all $p\in[1,\infty)$ and $k\ge 2$, if 	\[
	\ |\nabla^j\mathrm{Ric}|(x)\lesssim \lambda(r(x))
	^{\frac{2+j}{k-1}},\qquad 0\leq j \leq k-2,\] 
	and either 
\[ \mathrm{inj}(x)\gtrsim \lambda(r(x))^{-\frac1{k-1}},\qquad\text{ or }\qquad 
		|\mathrm{Riem}|(x)\lesssim\lambda(r(x))^{\frac2{k-1}};
		\]
\item[(ii)] $W^{2,2}(M)=W^{2,2}_0(M)$ if 
	\[
	\ \mathrm{Ric}(x)\gtrsim -\lambda(r(x))^2
	\]
	in the sense of quadratic forms, and $W^{k,2}(M)=W^{k,2}_0(M)$ for $k>2$ if 	\[
	\ |\nabla^j\mathrm{Riem}|(x)\lesssim \lambda(r(x))^{\frac{2+j}{k-1}},\qquad 0\leq j \leq k-3.\]
\end{itemize}		
\end{theorem} 
Observe that, for $k>2$, the assumptions in (i) and (ii) are skew. A noticeable feature of (ii) is that it requires a control on derivatives only up to the order $k-3$: for instance, 
\[
W^{3,2}(M)=W_0^{3,2}(M)\qquad \text{provided that}\quad |\mathrm{Riem}|(x)\lesssim \lambda(r(x)),
\] 
and, in particular, if $M$ has bounded sectional curvature (equivalently, bounded curvature operator). Quite surprisingly, for a long-time it remained unknown whether  $W_0^{k,p}(M)=W^{k,p}(M)$ on any complete Riemannian manifold or whether any assumption on $(M,g)$ was necessary in order to deduce the result. Very recently, an example has been found proving that
$W_0^{k,p}(M)\subseteq W^{k,p}(M)$ is a proper inclusion on certain manifolds with a very wild geometry, at least for $p\geq 2$; \cite{V-counter}. 
\medskip

In the first part of the present paper, we study the validity of the results in (ii) of Theorem \ref{t: IRV} in the range $p \in [1,2)$. For second order Sobolev spaces, we prove

\begin{theorem}\label{thm 2nd order}
	Let $(M,g)$ be a complete Riemannian manifold such that
	\begin{equation}\label{ipo_ricci_2nd}
	\Ricc(x)\gtrsim -\lambda(r(x))^2,
	\end{equation}
	in the sense of quadratic forms. Then, for all $p\in [1,2]$, we have 
	\[
	W^{2,p}_0(M)=W^{2,p}(M).
	\]	
\end{theorem}

\begin{remark}\label{lambdagen}
	\emph{This result still holds, with the same proof,  for slightly more general (yet more involved) choices for the function $\lambda$; for more details see \cite[Theorem 1.7 and Section 5]{IRV-HO}.}
\end{remark}

\begin{remark}\label{rmkRates}
\emph{The curvature of the example in \cite{V-counter} decays to $-\infty$ as $-r(x)^4$, and it seems difficult to refine the construction to make it decay at rate $-r(x)^{\alpha}$ for $\alpha$ close enough to $2$. Therefore, at present, there is a gap between the curvature decays in Theorem \ref{thm 2nd order} and in \cite{V-counter}. We anticipate that it would be interesting to produce a counterexample to Theorem \ref{thm 2nd order} in the range $p \in [1, 2]$ when \eqref{ipo_ricci_2nd} barely fails. A counterexample for $p>2$ will be given below.
}
\end{remark}

The proof in \cite{Bandara,IRV-HessCutOff} for the case $p=2$  breaks up into the following steps:
\begin{enumerate}
	\item in our assumptions, by \cite[Corollary 2.3]{BianchiSetti} there exists a family of Laplacian cut-off functions $\chi_R\in C^\infty_c(M)$   such that 
	\begin{itemize}
		\item $\chi_R= 1$ on $B_R(o)$,
		\item $|\nabla \chi_R|(x)\leq C\lambda^{-1}(r(x))$,
		\item $|\Delta \chi_R|\leq C$,
	\end{itemize} 
	for some constant $C>0$ independent of $R$;
	\item the above properties guarantee that $\|f\Delta\chi_R\|_{L^2(M)}\to 0$ as $R\to\infty$, for any $f\in W^{2,2}(M)$;
	\item using the Bochner formula, the latter step implies that $\|f\,\nabla^2\chi_R\|_{L^2(M)}\to 0$ as $R\to\infty$ for any $f\in W^{2,2}(M)$.
	\item this finally yields that $f\chi_R\to f$ in $\|\cdot\|_{W^{2,2}(M)}$.	
\end{enumerate}
The proof we provide here for the case $p<2$ follows the same line of thought. In this case one has to control $\|f\,\nabla^2\chi_R\|_{L^p(M)}$. Since the Bochner formula is modelled on $L^2$-norms, by a H\"older inequality we estimate $\|f\,\nabla^2\chi_R\|_{L^p(M)}$ in terms of $\||f|^{p/2}\,\nabla^2\chi_R\|_{L^2(M)}$ and use Bochner formula to control this latter. The main difficulty consists in estimating the remaining term involving $|f|^{p/2}$ and its derivatives. To this end, the case $p=1$ requires an \textsl{ad hoc} procedure, while for $p \in (1,2)$ we shall need a regularity lemma that, to the best of our knowledge, seems to be new. To state this latter, we first define the functional space
\[
\widetilde{W}^{2,p}(M) = \Big\{ f \in L^p(M) \ : \ \Delta f \in L^p(M) \ \text{distributionally}	\Big\}, 
\]
endowed with the norm 
\[
\|f\|_{\widetilde{W}^{2,p}(M)} = \|f\|_{L^p(M)} + \|\Delta f\|_{L^p(M)}.
\]	
It is important to notice that O. Milatovic, see \cite[Appendix A]{GuneysuPigola_AMPA}, proved that  $C^\infty_c(M)$ is dense in $\widetilde{W}^{2,p}(M)$ on any complete manifold $M$, independently of the behaviour of its curvatures.

\begin{lemma}\label{lem_reg}
	Let $M$ be a complete Riemannian manifold, and fix $p \in (1,\infty)$. If
	$f, |\nabla f|, \Delta f  \in L^{p}(M)$, then $|f|^{\frac{p}{2}} \in W^{1,2}(M)$,
	with the bound
	\begin{equation*}
		\ \left\|\nabla |f|^{\frac{p}{2}}\right\|_{L^{2}(M)}^2 \leq \frac{p^{2}}{4(p-1)}\| f \|_{L^p(M)}^{p-1} \|\Delta f\|_{L^p(M)}.
	\end{equation*}
	Moreover, when $1<p\leq 2$, if
	$f\in \widetilde{W}^{2,p}(M)$ then $|\nabla f|\in L^{p}(M)$ and 
	\begin{align}\label{eq_align_intro}
		\|\nabla f\|_{L^p(M)}^2 \le & \frac{4}{p^2} \|f\|_{L^p(M)}^{2-p} \left\| \nabla |f|^{\frac{p}{2}} \right\|_{L^2(M)}^2  \\ 
	\le & \frac{1}{p-1} \|f\|_{L^p(M)}\|\Delta f\|_{L^p(M)}  \qquad \forall \, f \in \widetilde{W}^{2,p}(M). \nonumber 
\end{align}	
\end{lemma}

 Notably, this lemma in the case $1<p\le 2$ refines the $L^p$-gradient estimate found by T. Coulhon and X. T. Duong, \cite{CD}, who showed that, for some constant $C_p$,
\begin{equation}\label{eq_CD_intro}
\|\nabla f\|_{L^p(M)}^2\leq C_p \| f\|_{L^p(M)}\|\Delta f\|_{L^p(M)},\quad\forall f\in C^\infty_c(M).
\end{equation}
At the same time, our proof avoids the use of tools from Harmonic Analysis, such as the mapping properties of the Littlewood-Paley's function. A further, different proof of \eqref{eq_CD_intro}, obtained from $L^p$-gradient estimates for the heat kernel, can be found in \cite{LiChen}. 

\begin{remark}\emph{
As pointed out in \cite[p.7]{CD}, a minor modification of the argument in \cite[Sec. 5]{CD_TAMS} shows the existence of manifolds (for instance, the connected sum of two copies of $\R^n$) for which \eqref{eq_CD_intro} fails for $p>n$. Indeed, these examples can be generalized to any $2<p\le n$ reasoning precisely as for Corollary \ref{cor_gionaludo} below. Note that, by Young's inequality, \eqref{eq_align_intro} implies the weaker $L^p$-gradient estimate
	\begin{equation}\label{CDadditive}
\|\nabla f\|_{L^p(M)} \le c \big( \|f\|_{L^p(M)} + \|\Delta f\|_{L^p(M)} \big).
\end{equation}
Because of \cite{ChengThalmaierThompson}, \eqref{CDadditive} is met for each $1<p<\infty$ if the Ricci curvature is bounded from below. A direct proof of this latter result can be found in \cite[Theorem 8.2]{Pigola_ArXiv}. Sufficient conditions for the validity of \eqref{eq_CD_intro}, more precisely of the stronger
\[
\|\nabla f\|_{L^p(M)}\leq C_p \|(-\Delta)^{1/2}f\|_{L^p(M)},
\]
have been investigated in \cite{bakry,CD,CD_TAMS,CarCouHas,Car,CCFR}. With no assumptions besides the completeness of $M$, the only $L^p$-gradient estimate that we are aware of is that in Theorem 2 in \cite{GuneysuPigola_AMPA}, where the authors prove the inequality
	\[
	\|\nabla f\|_{L^p(M)}^2 \le C_p \|f\|_{L^p(M)} \Big( \|\Delta f\|_{L^p(M)} + \max\{0, p-2\} \|\nabla^2 f\|_{L^p(M)}\Big),
	\]
for $f\in L^{p}(M)$ with $|\nabla^{2}f|\in L^{p}(M)$.
}\end{remark}

We can reproduce the same scheme of proof introduced for $k=2$ also for higher orders. The main tool will be a Weitzenb\"ock formula due to J. H. Sampson applied to the totally symmetrized $(k-1)$-th covariant derivative of some special higher order cut-off functions; see Section \ref{section-HO} for precise definitions. This point of view has been recently exploited in \cite[Section 5]{IRV-HO} in the case $p=2$, finally leading to the result described in Theorem \ref{t: IRV}(ii). Combining this latter technique with our regularity lemma, we are able to deal with the full range $p\in[1,2]$ and prove the following
\begin{theorem}\label{teo_higher}
Let $(M,g)$ be a complete Riemannian manifold such that, for some integer $k>2$,
\[
\ |\nabla^{j}\mathrm{Riem}|(x)\lesssim \lambda(r(x))^{\frac{2+j}{k-1}},\qquad 0\leq j\leq k-3,
\]
with $\lambda$ as in \eqref{lambda}. Then, for all $p\in[1,2]$, we have
\[
\ W^{k,p}(M)=W_{0}^{k,p}(M).
\]
\end{theorem}
\medskip
 
The next step is to understand if one could obtain the equality $W_0^{2,p}(M) = W^{2,p}(M)$ under a lower Ricci curvature bound for $p>2$. In the second part of the paper, we show that the answer is negative even if one assumes a lower bound on the sectional curvature. We obtain

\begin{theorem}\label{th_striking}
For all $n\ge 2$ and $p>2$, there exists a complete, $n$-dimensional Riemannian manifold $Q$ with $\mathrm{Sec}\geq -1$ and satisfying
	\begin{equation}\label{eq_th_striking}
W^{k,p}(Q) \neq W_0^{k,p}(Q) \qquad \text{for each } \, k \ge 2.
	\end{equation}
\end{theorem}

The manifold $Q$ has at least two ends with finite volume, and can be constructed to have finite volume as well as to have infinite volume. Its existence will be a consequence of the following result, where we produce a ``block" that can be attached to any smooth manifold. More precisely, we prove
\begin{theorem}\label{Counterex}
	For all $n\geq 2$ and all $p>n$ there exists a complete $n$-dimensional Riemannian manifold $(M,g)$ with sectional curvature $\Sec \ge -1$, with a  distinguished relatively compact, open subset $\vm$ diffeomorphic to $\mathbb{R}^{n}$, such that the following holds: for each $n$-dimensional Riemannian manifold $(N,\bar g)$, and each relatively compact set $\vm' \subset N$ that is diffeomorphic to $\mathbb{R}^{n}$, the connected sum $M \sharp N$ obtained by gluing along $\vm$ and $\vm'$ and keeping the original metric outside of $\vm, \vm'$ satisfies
\[
W^{k,p}(M \sharp N) \neq W_0^{k,p}(M \sharp N) \qquad \text{for each } \, k \ge 2.
\]
In particular, if $N$ has sectional curvature bounded from below the same holds for $M \sharp N$.
\end{theorem}

\begin{remark}
\emph{As we shall see, $M$ is topologically a product $\mathbb{S}^{n-1} \times \R$ and has finite volume. 
}	
\end{remark}

If now, given $n\ge 2$ and $p>2$, we select a surface $M \sharp N$ as in Theorem \ref{Counterex} applied in dimension $2$, and a compact boundaryless manifold $Y$ of dimension $n-2$, then it is easy to show that the Riemannian manifold $Q\doteq (M\sharp N)\times Y$ satisfies \eqref{eq_th_striking}; see Subsection \ref{striking}. Moreover, if $N$ is complete and has sectional curvature bounded from below the same holds for $Q$.

Indeed, this  last observation allows to extend to the range $p \in (2,n]$ two other counterexamples that were previously known only in the case $p > n$. In order to introduce them, let us first  note that there is a tight relation between the density of compactly supported functions in $W^{2,p}(M)$ and the validity of  a global $L^p$-Calder\'on-Zygmund inequality
\begin{equation}\label{CZ(p)}\tag{CZ$_p$}
\|\nabla^2 f\|_{L^p(M)}\leq C (\| f\|_{L^p(M)}+\|\Delta f\|_{L^p(M)}),\quad\forall f\in C^\infty_c(M).
\end{equation}
Indeed, as illustrated in \cite[Proposition 4.7]{Pigola_ArXiv}, 
	\begin{equation}\label{CZdensity}
	\text{$M$ supports \eqref{CZ(p)}} \qquad \Longrightarrow \qquad W_0^{2,p}(M)=W^{2,p}(M),
	\end{equation} 
while the converse is not always true; see Theorem A and the subsequent discussion in \cite{MV}. Therefore, counterexamples to the density of $C^\infty_c(M)$ in $W^{2,p}(M)$ have to be searched among those manifolds that do not support \eqref{CZ(p)}. The existence of such manifolds has first been proved in \cite{GuneysuPigola,Li}. However, in these constructions the curvature is not lower bounded. Very recently, the first example of a complete non-compact manifold with non-negative sectional curvature on which \eqref{CZ(p)} fails for $p>n$ has been presented in \cite{MV} by L. Marini and the fourth author. Their example confirms a strong indication suggested by a work by G. De Philippis and J. N\'u\~nez-Zimbron, \cite[Cor. 1.3]{dePNZ}, where it is proved that for $p>n$ it is not possible to construct a Calder\'on-Zygmund theory on compact manifolds with constants depending only on (a diameter upper bound and) a lower sectional curvature bound. The same trick that allows us to deduce Theorem \ref{th_striking} from Theorem \ref{Counterex} also enables us to extend the above mentioned results in \cite{MV,dePNZ} to the full range $p>2$:

\begin{corollary}\label{cor_gionaludo}
For any $p>2$, there exists a complete $n$-dimensional Riemannian manifold $Q$ with $\Sec \ge 0$ that does not satisfy \eqref{CZ(p)}. Precisely, if $M^2$ is a $2$-dimensional manifold as constructed in \cite{MV}, and if $Y^{n-2}$ is a compact manifold with $\Sec \ge 0$, then one can take $Q = M^2 \times Y^{n-2}$.
\end{corollary}

\begin{remark}
\emph{Clearly, by \eqref{CZdensity} the counterexample in Theorem \ref{th_striking} is also a counterexample to \eqref{CZ(p)}. The extra information in Corollary \ref{cor_gionaludo} is the possibility to construct an example with $\Sec \ge 0$, in particular, by the lower volume bound given by Calabi-Yau theorem (see e.g. \cite{Yau}), each end of $Q$ has infinite volume.  
}
\end{remark}

\begin{corollary}\label{cor_dePNZ}
Let $n \in \NN$, $D \ge 2$ and $p>2$. Then there exist sequences of $n$-dimensional complete Riemannian manifolds $(Q_k,g_k)$ with $\mathrm{diam}_{g_k}(Q_k) \le D$, $\Sec_{g_k} \ge 0$ and of smooth functions $f_k \in C^\infty(Q_k)$ such that
	\[
	\|f_k\|_{L^p(Q_k)} + \|\Delta f_k\|_{L^p(Q_k)} = 1
	\]
but
	\[
	\lim_{k \to \infty} \|\nabla^2 f_k\|_{L^p(Q_k)} = \infty.
	\]
\end{corollary}

We now explain the strategy to prove Theorem \ref{Counterex}. Contradicting \eqref{CZ(p)} on $M$ is, in principle, easier than contradicting $W^{2,p}_0(M)=W^{2,p}(M)$. Indeed, for the former, it is enough to prove that, for any given $C>0$, there exists at least one compactly supported function $f$ for which \eqref{CZ(p)} with constant $C$ fails. In particular the construction can be localized in a given region of $M$. On the other hand, to disprove the density of $C^\infty_c(M)$ in $W^{2,p}(M)$ one needs to handle with any possible compactly supported approximation of a given function. A way to overcome this problem has been proposed in \cite{V-counter}, which contains the first (and so far unique) example of $W^{2,p}_0(M)\neq W^{2,p}(M)$ whenever $p\ge 2$. 
  Namely, one can consider a complete manifold $(M,g)$ with two ends $E_+$ and $E_-$ and finite volume, so that it is possible to choose a function $f\in W^{2,p}(M)$ which attains two different constant values (say $1$ and $-1$) on each end. Accordingly, to check that $f$ has no compactly supported approximations, it is enough to prove that the $W^{2,p}$-norm of any function $F$  which is identically $-1$ and $1$ on the two ends cannot be arbitrarily close to zero. In the presence of a constant lower curvature bound, our strategy can roughly  be summarized as follows. First, we construct a suitable Alexandrov space $(M, d_{\infty})$ with finite volume, $\mathrm{Sec}\geq -1$, and a dense set of sharp singular points. We consider an exhaustion $U_{j}$ of $M$. On each annulus $U_{j}\setminus U_{j-1}$, inspired by \cite{dePNZ, MV} we prove the existence of a family of metrics $\{\sigma_{j,k}\}_k$ that GH-converge to $d_{\infty}$, and then we suitably select a function $k : \NN \to \NN$ to produce a global metric $g$ on $M$ that equals $\sigma_{j,k(j)}$ on $U_j \setminus U_{j-1}$ and has the following property: any function $F$ with $\left\|F\right\|_{W^{2,p}(U_{j+1})}\leq  1$ has to be $C^0$-close to a constant on $U_{j}\setminus U_{j-1}$, in a quantitative way. In particular, if $\left\|F\right\|_{W^{2,p}(M)}$ is small enough, then $F$ cannot attain values $-1$, $1$ on the two ends, as required.
\medskip 

We conclude this introduction with a list of related questions for future research.

\begin{itemize}
\item[-] Is it possible to construct a complete manifold with $\mathrm{Sec} \ge 0$ for which $W^{2,p}_0(M) \neq W^{2,p}(M)$ for some $p > 2$? What about if we weaken the curvature assumption to $\mathrm{Ric}\ge 0$? In both of the cases, by Calabi and Yau's theorem all ends have infinite volume, so the construction in Theorem \ref{Counterex} cannot be adapted in a straightforward way. 
\item[-] The manifolds constructed in Theorem \ref{th_striking} and Corollary \ref{cor_gionaludo} are Riemannian products and, in particular, they have nontrivial topology. It would still be interesting to produce counterexamples in the range $p \in (2,n]$ by generalizing, if possible, the technique in \cite{dePNZ}. This may lead, for instance, to counterexamples to \eqref{CZ(p)} for $p \in (2,n]$ on contractible manifolds. 
\item[-] Referring to Remark \ref{rmkRates}, are the decay rates assumed for the curvatures considered in Theorem \ref{thm 2nd order} and Theorem \ref{teo_higher} sharp? It seems reasonable to conjecture so, up to lower order terms.
\item[-] Does a complete manifold with positive injectivity radius and $\mathrm{Ric} \ge - 1$ satisfy $W^{2,p}_0(M) = W^{2,p}(M)$ for each $p \in [1,\infty)$? Recall that Theorem \ref{t: IRV} answers affirmatively under the conditions $|\mathrm{Ric}| \lesssim \lambda^2(r)$ and ${\rm inj}  \gtrsim \lambda(r)^{-1}$. If $|\mathrm{Ric}| \lesssim 1$, is a decay assumption on the injectivity radius necessary?
\end{itemize}


\section{Density when $p \in [1,2]$}

\subsection{The regularity lemma}
	
\begin{lemma}\label{lem_reg}
Let $M$ be a complete Riemannian manifold, and fix $p \in (1,\infty)$. Let $o\in M$ be some fixed origin, and let us denote by $B_{r}$ the geodesic balls of radius $r$ centered at $o$. If $f \in W^{2,p}_\loc(M)$ then $|f|^{\frac p2} \in W^{1,2}_\loc(M)$ and, for each $0 < R < r$,  
	\begin{equation}\label{eq_local_reg}
	\frac{4(p-1)}{p^2} \left\| \nabla |f|^{\frac{p}{2}} \right \|_{L^2(B_{R})}^2 \le \| f \|_{L^p(B_{r})}^{p-1} \left( \frac{1}{r-R} \|\nabla f\|_{L^p(B_{r})} + \|\Delta f\|_{L^p(B_{r})} \right).	
	\end{equation}
In particular, 
	\begin{equation}\label{eq_gradient_reg}
	\begin{array}{lcl}
	f, |\nabla f|, \Delta f  \in L^{p}(M) & \qquad \Longrightarrow \qquad & |f|^{\frac{p}{2}} \in W^{1,2}(M),
	\end{array}
	\end{equation}
with the bound
	\begin{equation}\label{eq_bello}
	\ \left\|\nabla |f|^{\frac{p}{2}}\right\|_{L^{2}(M)}^2 \leq \frac{p^{2}}{4(p-1)}\| f \|_{L^p(M)}^{p-1} \|\Delta f\|_{L^p(M)}.
	\end{equation}
Moreover, if $1<p\leq 2$,
	\begin{equation*}
	\begin{array}{lcl}
	f\in \widetilde{W}^{2,p}(M)  &\qquad \Longrightarrow \qquad & |\nabla f|\in L^{p}(M),\quad|f|^{\frac{p}{2}} \in W^{1,2}(M),
	\end{array}
	\end{equation*}
	and
	\begin{align}\label{CDEstimate}
	\|\nabla f\|_{L^p(M)}^2 \le & \frac{4}{p^2} \|f\|_{L^p(M)}^{2-p} \left\| \nabla |f|^{\frac{p}{2}} \right\|_{L^2(M)}^2  \\ 
	\le & \frac{1}{p-1} \|f\|_{L^p(M)}\|\Delta f\|_{L^p(M)}  \qquad \forall \, f \in \widetilde{W}^{2,p}(M). \nonumber 
	\end{align}
	\end{lemma}

\begin{remark}\label{rem_CZ}
\emph{In view of the validity of a local Calder\'on-Zygmund inequality, we note that the assumption $f\in W^{2,p}_{\mathrm{loc}}(M)$ is equivalent to the assumption $f, |\nabla f|,  \Delta f \in L^{p}_{\mathrm{loc}}(M)$.}
\end{remark}

\begin{proof}
Let us first assume that $f \in C^\infty(M)$. Clearly, $|f|^{\frac p2} \in L^2_\loc(M)$. Let $\varphi$ be a linear cut-off function with ${\rm supp}\, \varphi \subset \overline{B_{r}}$, $\varphi \equiv 1$ on $B_R$ and $|\nabla \varphi| \le 1/(r-R)$. For $\eps>0$, we compute
	\begin{align*}
	\disp - \int_M \langle \nabla(f^2 + \eps)^\frac{p}{2}, \nabla \varphi \rangle=& \disp \int_M \varphi \Delta (f^2 + \eps)^\frac{p}{2} \\
	=& \disp p \int_M \varphi (f^2 + \eps)^\frac{p-2}{2}[ f \Delta f + |\nabla f|^2] + p(p-2)\int_M \varphi (f^2 + \eps)^\frac{p-4}{2}f^2 |\nabla f|^2 \\
	= & \disp p \int_M \varphi (f^2 + \eps)^\frac{p-2}{2} f \Delta f + p \int_M \varphi (f^2 + \eps)^\frac{p-2}{2}|\nabla f|^2 \\
	&+ p(p-2)\int_M \varphi (f^2 + \eps)^\frac{p-4}{2}f^2 |\nabla f|^2 \\
	\ge & \disp p \int_M \varphi (f^2 + \eps)^\frac{p-2}{2} f \Delta f + p(p-1) \int_M \varphi (f^2 + \eps)^\frac{p-4}{2} f^2 |\nabla f|^2. 
\end{align*}
On the one hand, 
	\begin{align*}
\disp \left | \int_M \langle \nabla(f^2 + \eps)^\frac{p}{2}, \nabla \varphi \rangle \right| = & \disp p \left|\int_M (f^2 + \eps)^\frac{p-2}{2} f \langle \nabla f, \nabla \varphi \rangle \right| \\
\le & \disp p \int_M (f^2 + \eps)^\frac{p-1}{2} |\nabla f| |\nabla \varphi| \\
\le & \disp \frac{p}{r-R} \left( \int_{B_{r}} (f^2 + \eps)^{\frac{p}{2}} \right)^{\frac{p-1}{p}} \left( \int_{B_{r}} |\nabla f|^p \right)^{\frac{1}{p}},
	\end{align*}
on the other hand, 
	\[
	\left| \int_M \varphi (f^2 + \eps)^\frac{p-2}{2}f \Delta f\right| \le \int_M \varphi (f^2 + \eps)^\frac{p-1}{2} |\Delta f| \le \left( \int_{B_{r}} (f^2 + \eps)^\frac{p}{2} \right)^{\frac{p-1}{p}} \left( \int_{B_{r}} |\Delta f|^p \right)^{\frac{1}{p}}. 	
	\]
Summarizing, 
	\[
	\begin{array}{lcl}
	\disp \frac{4(p-1)}{p^2} \int_{B_R} \left| \nabla(f^2+\eps)^{\frac{p}{4}} \right|^2 & = & \disp (p-1) \int_{B_R} (f^2 + \eps)^\frac{p-4}{2} f^2 |\nabla f|^2 \\[0.5cm]
	& \le & \disp \| \sqrt{f^2 + \eps} \|_{L^p(B_{r})}^{p-1} \left( \frac{1}{r-R} \|\nabla f\|_{L^p(B_{r})} + \|\Delta f\|_{L^p(B_{r})} \right).
	\end{array}
	\]
Hence, $\{(f^2+\eps)^{\frac{p}{4}}\}$ is uniformly bounded in $W^{1,2}(B_R)$ and pointwise convergent to $|f|^{p/2}$. By a standard result (\cite[Lemma 6.2, p.16]{Friedman}), $|f|^{p/2} \in W^{1,2}(B_R)$ and $\nabla (f^2+\eps)^{p/4} \rightharpoonup \nabla |f|^{p/2}$ weakly on $B_R$, thus 
	\begin{align*}
	\frac{4(p-1)}{p^2} \int_{B_R} \left| \nabla |f|^{\frac{p}{2}} \right|^2 \le & \frac{4(p-1)}{p^2} \liminf_{\eps \to 0} \int_{B_R} \left| \nabla(f^2+\eps)^{\frac{p}{4}} \right|^2 \\
	\le & \| f \|_{L^p(B_{r})}^{p-1} \left( \frac{1}{r-R} \|\nabla f\|_{L^p(B_{r})} + \|\Delta f\|_{L^p(B_{r})} \right).
	\end{align*}
We now claim that \eqref{eq_local_reg} holds for $f \in W^{2,p}_\loc(M)$. Having chosen such $f$, by the Meyers-Serrin-type theorem in \cite{GGP-AccFinn} there exists $\{f_j\} \subset C^\infty(M)$ such that $f_j \to f$ in $W^{2,p}(B_{r})$ and pointwise almost everywhere. Applying \eqref{eq_local_reg} to $f_j$ shows that $\{|f_j|^{p/2}\}$ is uniformly bounded in $W^{1,2}(B_R)$, so by weak compactness and pointwise convergence we deduce that $|f_j|^{p/2} \rightharpoonup |f|^{p/2}$ in $W^{1,2}(B_R)$. Evaluating \eqref{eq_local_reg} on $f_j$, passing to limits and using the weak lower semicontinuity of gradients, we deduce \eqref{eq_local_reg} for each $f \in W^{2,p}_\loc(M)$, as claimed.

Assume that $f, |\nabla f|, \Delta f \in L^p(M)$. Then, by Remark \ref{rem_CZ} $f \in W^{2,p}_\loc(M)$ and thus letting $r = 2R \to \infty$ in \eqref{eq_local_reg} we readily deduce \eqref{eq_bello}.

Next, we examine the case $f \in \widetilde{W}^{2,p}(M)$. We first consider $f \in C^\infty_c(M)$, since this latter space is dense in $\widetilde{W}^{2,p}(M)$ by \cite[Appendix A]{GuneysuPigola_AMPA}. If $1<p\le 2$, by H\"older's inequality, Stampacchia's theorem and \eqref{eq_local_reg} we get
	\begin{align}\label{eq_align}
	\| \nabla f \|^2_{L^p(B_R)} = & \| \nabla f \|^2_{L^p(B_R \cap \{|f|>0\})} \\
	\le & \|f \|_{L^p(B_R \cap \{|f|>0\})}^{2-p} \left\| f^{\frac{p-2}{2}}\nabla f \right\|_{L^2(B_R \cap \{|f|>0\})}^2 \nonumber \\
	= & \frac{4}{p^2} \|f\|_{L^p(B_R)}^{2-p} \left \| \nabla |f|^\frac{p}{2} \right\|_{L^2(B_R \cap \{|f|>0\})}^2 \nonumber  \\
	\le & \frac{4}{p^2} \|f\|_{L^p(B_r)}^{2-p} \left \| \nabla |f|^\frac{p}{2} \right\|_{L^2(B_R)}^2 \nonumber \\
	\le & \frac{1}{p-1} \|f\|_{L^p(B_r)} \left( \frac{1}{r-R} \|\nabla f\|_{L^p(B_{r})} + \|\Delta f\|_{L^p(B_{r})} \right). \nonumber 
\end{align}
If $f \in \widetilde{W}^{2,p}(M)$, take $\{f_j\} \subset C^\infty_c(M)$ with $f_j \to f$ and $\Delta f_j \to \Delta f$ in $L^p(M)$. Applying \eqref{eq_align} to $f_j$ and to $f_j - f_i$, letting $r = 2R \to \infty$ and then $j \to \infty$ we deduce that $\nabla f_j \to \nabla f \in L^p(M)$. In particular, by \eqref{eq_gradient_reg} the function $f$ satisfies \eqref{eq_bello}. The same computations as in \eqref{eq_align} can therefore be performed with $R, r = \infty$, leading to \eqref{CDEstimate}.
\end{proof}


\subsection{Density for order $2$}

\begin{proof}[Proof of Theorem \ref{thm 2nd order}]
	 For $R\gg1$, let $\chi_R\in C^\infty_c(M)$ be a family of Laplacian cut-off functions such that 
\begin{itemize}
	\item $\chi_R= 1$ on $B_R(o)$,
	\item $|\nabla \chi_R|\leq C\lambda^{-1}(r)$ and $\|\nabla\chi_R\|_\infty\leq C \lambda^{-1}(R)$,
	\item $|\Delta \chi_R|\leq C$,
\end{itemize} 
 for some constant $C>0$ independent of $R$. Such a family has been constructed in (the proof of)  \cite[Corollary 5.2]{IRV-HO}.
As usual, first note that $C^{\infty}(M)\cap W^{2,p}(M)$ is dense in $W^{2,p}(M)$ (see for instance \cite{GGP-AccFinn}). Given a smooth $f\in W^{2,p}(M)$, define $f_{R}\doteq\chi_{R}f$. We get that 
 \begin{align}
 \|(f_R-f)\|_{L^p}&=\|((1-\chi_R)f)\|_{L^p}\label{conv1}\\
 \|\nabla (f_R-f)\|_{L^p}&\leq \|f\nabla\chi_R\|_{L^p}+\|(1-\chi_R)\nabla f\|_{L^p}\label{conv2}\\
 \|\nabla^2(f_R-f)\|_{L^p}&\leq  2\||\nabla \chi_R| |\nabla f|\|_{L^p}+\|(1-\chi_R)\nabla^2 f\|_{L^p}+\|f\nabla^2\chi_R\|_{L^p}\label{conv3}
 \end{align}
Both $(1-\chi_R)$ and $\nabla\chi_R$ are uniformly bounded and supported in $M\setminus B_R(o)$. Since $f\in W^{2,p}(M)$ this permits to conclude that all the terms at the RHS of \eqref{conv1}, \eqref{conv2} and \eqref{conv3} except the last one tend to $0$ as $R\to\infty$. Concerning 
$\|f\nabla^2\chi_R\|_{L^p}$, first observe that $p\leq 2$ and H\"older's inequality imply
\begin{equation}\label{f: holder}
\int_M |f|^p |\nabla^2\chi_R|^p \leq \left(\int_M |f|^p |\nabla^2\chi_R|^2\right)^{\frac p2} \left(\int_M |f|^p \right)^{\frac{2-p}{2}}.
\end{equation}
 Accordingly, to conclude it is enough to show that
\[ \int_M |f|^p |\nabla^2\chi_R|^2 \to 0 \qquad \text{as } \, R \to \infty.
\]
Inserting into Bochner formula
\begin{equation*}
\frac{1}{2} \diver\big( \nabla |\nabla u|^2 \big)  = |\nabla^2 u|^2 + \mathrm{Ric}(\nabla u, \nabla u) + \langle \nabla \Delta u, \nabla u \rangle \qquad \forall \, u \in C^\infty(M)
\end{equation*}
the function $u=\chi_R$, multiplying by $|f|^p$ and integrating over $M$ gives
\begin{align*}
\frac{1}{2} \int_M |f|^p \diver\big( \nabla |\nabla \chi_R|^2 \big)  &= \int_M |f|^p|\nabla^2 \chi_R|^2 + \int_M |f|^p\mathrm{Ric}(\nabla \chi_R, \nabla \chi_R) + \int_M |f|^p\langle \nabla \Delta \chi_R, \nabla \chi_R \rangle.
\end{align*}
Applying Stokes' theorem to the first and the last integral, we get
\begin{align}\label{bochner}
\int_M |f|^p|\nabla^2 \chi_R|^2
=& - \frac{1}{2} \int_M \langle \nabla (|f|^p), \nabla |\nabla \chi_R|^2 \rangle  - \int_M |f|^p\mathrm{Ric}(\nabla \chi_R, \nabla \chi_R) \\
&+ \int_M |f|^p|\Delta \chi_R|^2+\int_M \Delta \chi_R\langle \nabla (|f|^p), \nabla \chi_R \rangle.\nonumber
\end{align}
First, note that
\begin{align}\label{lapl_int}
\int_M \Delta \chi_R\langle \nabla (|f|^p), \nabla \chi_R \rangle\leq C \int_{M\setminus B_R(o)} |\nabla (|f|^p)| \leq C p \left(\int_{M\setminus B_R(o)} |\nabla |f||^p\right)^{1/p}\left(\int_{M\setminus B_R(o)} |f|^{p}\right)^{(p-1)/p}.
\end{align}
Similarly
\[
 - \int_M |f|^p\mathrm{Ric}(\nabla \chi_R, \nabla \chi_R) \leq \int_{M\setminus B_R(o)} C\lambda^{-2}(r)\lambda^2(r)|f|^p \leq C \int_{M\setminus B_R(o)} |f|^p 
\]
and 
\[
\int_M |f|^p|\Delta \chi_R|^2 \leq C \int_{M\setminus B_R(o)} |f|^p. 
\]
In particular, 
\begin{equation}\label{int_2}
- \int_M |f|^p\mathrm{Ric}(\nabla \chi_R, \nabla \chi_R) + \int_M |f|^p|\Delta \chi_R|^2+\int_M \Delta \chi_R\langle \nabla (|f|^p), \nabla \chi_R \rangle \to 0
\end{equation}
as $R \to \infty$ for $f\in W^{1,p}(M)$. 
Inserting \eqref{int_2} in \eqref{bochner} we deduce that, in order to prove that
\[
\int_M |f|^p|\nabla^2 \chi_R|^2\to 0
\]
as $R\to \infty$, it is enough to show that
\begin{equation}\label{eq_claimedineq}
\limsup_{R\to\infty}-\frac 12 \int_M \langle \nabla (|f|^p), \nabla |\nabla \chi_R|^2 \rangle - c\int_M |f|^p|\nabla^2 \chi_R|^2\leq 0,
\end{equation}
for some $c<1$ independent of $R$.
We first suppose that $p\in(1,2)$. By Kato and Young's inequalities we have that 
\begin{align}\label{int_1}
-\frac 12 \int_M \langle \nabla (|f|^p), \nabla |\nabla \chi_R|^2 \rangle 
&\leq  2 \int_M |f|^{\frac p2} \cdot|\nabla |f|^{\frac p2}|\cdot |\nabla \chi_R| \cdot |\nabla |\nabla \chi_R||\\
&\leq \frac 12 \int_M |f|^p|\nabla|\nabla \chi_R||^2 + 4 \int_{M\setminus B_R(o)} |\nabla |f|^{\frac p2}|^2\cdot |\nabla \chi_R|^2\nonumber\\
&\leq \frac 12 \int_M |f|^p|\nabla^2\chi_R|^2 + 4 \int_{M\setminus B_R(o)} |\nabla |f|^{\frac p2}|^2,\nonumber
\end{align}
where the last integral is finite and goes to $0$ as $R\to \infty$ due to Lemma \ref{lem_reg}. Hence, \eqref{eq_claimedineq} holds with $c = 1/2$.

In order to deal with the case $p=1$, we prove that the first addendum in \eqref{eq_claimedineq} vanishes as $R \to \infty$, so \eqref{eq_claimedineq} holds for every $c >0$. First, observe that, for each $R$,
\[
\int_M \langle \nabla |f|, \nabla |\nabla \chi_R|^2 \rangle=\lim_{\eps\to 0}\int_M \langle \nabla ((f^2+\eps)^{1/2}), \nabla |\nabla \chi_R|^2 \rangle.\]
Indeed, 
\begin{equation}\label{f: p1 1}
\left|\langle \nabla ((f^2+\eps)^{1/2}), \nabla |\nabla \chi_R|^2 \rangle\right|
\leq \left|\nabla |\nabla \chi_R|^2\right|\frac{|f|\,|\nabla f|}{(f^2+\eps)^{1/2}}\leq \left|\nabla |\nabla \chi_R|^2\right|\,|\nabla f|,
\end{equation}
so that Lebesgue's dominated convergence theorem applies. Next, for every $g \in C^1_c(M)$
\begin{align}\label{f: p1 2}
-\int_M \langle \nabla ((f^2+\eps)^{1/2}), \nabla g \rangle
&=\int_M \Delta((f^2+\eps)^{1/2})g\\
&= \frac 12 \int_M\frac{\Delta(f^2+\eps)}{(f^2+\eps)^{1/2}}g - \frac 14 \int_M  \frac{\left|\nabla(f^2+\eps)\right|^2}{(f^2+\eps)^{3/2}}g\nonumber\\
&= \int_M \frac{f\,\Delta f}{(f^2+\eps)^{1/2}}g + \int_M  \frac{\eps \left|\nabla  f\right|^2}{(f^2+\eps)^{3/2}}g,\nonumber 
\end{align}
and rearranging, we get 
\begin{align}\label{f: p1 chiR}
\int_M  \frac{\eps \left|\nabla  f\right|^2}{(f^2+\eps)^{3/2}}g &= -\int_M \langle \nabla ((f^2+\eps)^{1/2}), \nabla g \rangle -  \int_M \frac{f\,\Delta f}{(f^2+\eps)^{1/2}}g\\ 
&\leq  \int_M |\nabla f|\,| \nabla g| +  \int_M |\Delta f||g|.\nonumber
\end{align}
Since $|\nabla f|\in L^1(M)$, applying \eqref{f: p1 chiR} with $g= \chi_R$ and letting $R \to \infty$, we get
\begin{equation}\label{globalinte}
\int_M  \frac{\eps \left|\nabla  f\right|^2}{(f^2+\eps)^{3/2}}  
\leq   \int_M |\Delta f|.
\end{equation}
On the other hand, applying \eqref{f: p1 2} with $g = |\nabla \chi_R|^2$ and using \eqref{globalinte} we infer 

\begin{align*}
- \int_M \langle \nabla |f|, \nabla |\nabla \chi_R|^2 \rangle 
&= - \lim_{\eps\to 0}\int_M \langle \nabla ((f^2+\eps)^{1/2}), \nabla |\nabla \chi_R|^2 \rangle \\ \nonumber
&= \lim_{\eps \to 0} \left[ \int_M \frac{f\,\Delta f}{(f^2+\eps)^{1/2}}|\nabla \chi_R|^2 + \int_M  \frac{\eps \left|\nabla  f\right|^2}{(f^2+\eps)^{3/2}}|\nabla \chi_R|^2\right] \\ \nonumber 
&\leq 2\|\nabla\chi_R\|^2_\infty\,\int_{M} |\Delta f|
,
\end{align*}
which vanishes as $R\to\infty$ because of the properties of $\nabla \chi_R$, as claimed.
\end{proof}

\subsection{Density for higher orders}\label{section-HO}

We shall first recall a few facts about Sampson's Weitzenb\"ock formula for symmetric tensors. Given an $n$-dimensional Riemannian manifold $(M, g)$, consider a tensor bundle $E\to M$ with $n$-dimensional fibers endowed with an inner product induced by the metric $g$ and a compatible connection $\nabla$ induced by the Levi-Civita connection on $M$. A Lichnerowicz Laplacian $\Delta_{L}$ for $E$ is a second order differential operator acting on the space of smooth sections $\Gamma(E)$ of the form
\[
\ \Delta_{L}=\Delta_{B}+c\,\mathfrak{Ric},
\]
for $c$ a suitable constant. Here  $\Delta_{B}\doteq-\mathrm{tr}_{12}(\nabla^{2})=\nabla^{*}\nabla$ is the Bochner Laplacian (with $\nabla^{*}$ denoting the formal $L^{2}$-adjoint of $\nabla$) and $\mathfrak{Ric}$ is a smooth symmetric endomorphism of $\Gamma(E)$ which is called the Weitzenb\"ock curvature operator. 
As an example, note that when $T$ is a $(0,k)$-tensor then
\[
\ \mathfrak{Ric}(T)(X_{1},\ldots, X_{k})=\sum_{i=1}^{k}\sum_{j}(R(E_{j},X_{i})T)(X_{1},\ldots,E_{j},\ldots, X_{k}),
\]
with $\left\{E_{i}\right\}$ a local orthonormal frame and 
\[
\ R(X, Y) \doteq \nabla^{2}_{X,Y}-\nabla^{2}_{Y,X}=\nabla_{X}\nabla_{Y}-\nabla_{Y}\nabla_{X}-\nabla_{[X,Y]},
\]
which may act on any tensor field. It is important to notice that the Weitzenb\"ock curvature term for $(0,k)$-tensors can actually be estimated in terms of the curvature operator $\mathcal{R}$ of $M$. Indeed, if $\mathcal{R}\geq \alpha$, for some constant $\alpha<0$ then $g(\mathfrak{Ric}(T),T)\geq \alpha C|T|^2$, where $C$ depends only on $k$; see \cite[Corollary 9.3.4]{petersen}. This key feature of Lichnerowicz Laplacians permits to use geometric assumptions in estimation results.

As beautifully illustrated in \cite[Chapter 9]{petersen} there are several natural Lichnerowicz Laplacians on Riemannian manifolds. A very classical one is the Hodge-Laplacian $\Delta_{H}$ acting on exterior differential forms, for which the Weitzenb\"ock identity takes the form
\begin{equation}\label{f: Weitz}\Delta_{H}\omega\doteq (d\delta + \delta d) \omega  = \Delta_{B}\omega+\mathfrak{Ric}(\omega).\end{equation}
The Bochner identity which we used in the previous section precisely comes out from this formula evaluated on the 1-form $d\chi_R$. When considering the higher order case $k>2$, one may be tempted to use \eqref{f: Weitz} applied to the $(k-1)$-th covariant derivative of suitable cut-off functions. Unfortunately,  this latter is not at all skew-symmetric. However, it is at least almost symmetric, meaning that it can be decomposed as a totally symmetric principal part plus other terms involving derivatives of order at most $k-2$. This fact led the authors of \cite{IRV-HO} to consider a different Lichnerowicz Laplacian, acting on totally symmetric covariant tensors of any order, which was originally introduced by Sampson in \cite{Sampson} and that we now recall. Let $T^{(0,k)}(M)$ and $S^{(0,k)}(M)$ be, respectively, the bundle of $k$-covariant tensor and its subbundle of totally symmetric ones. Consider the operator $D_{S}:\Gamma S^{(0,k-1)}(M)\to\Gamma S^{(0,k)}(M)$ acting on $h\in\Gamma S^{(0,k-1)}(M)$ by
\begin{align*}
(D_{S}h)(X_{0}, \ldots, X_{k-1}) \doteq k s_{k}(\nabla h)(X_{0}, \ldots, X_{k-1}),
\end{align*}
where we are denoting  by $s_{k}$  the symmetrization operator, i.e. the projection of $T^{(0,k)}(M)$ onto  $S^{(0,k)}(M)$, that we shortly denote with a superscript $S$. Namely, 
\[
T^{S}(X_{1},\ldots,X_{k}) \doteq s_{k}(T)(X_{1},\ldots,X_{k}) = \frac{1}{k!}\sum_{\sigma\in\Pi_{k}}T(X_{\sigma(1)},\ldots,X_{\sigma(k)}) \qquad \forall \, T \in T^{(0,k)}(M).
\] 
Note that
	\begin{equation}\label{eq_tensors1}
|T^S| \le |T| .
	\end{equation}
The formal  $L^{2}$- adjoint of $D_{S}$ is $D_{S}^{*}: \Gamma S^{(0,k)}(M)\to\Gamma S^{(0,k-1)}(M)$  which acts on $\tilde{h}\in\Gamma S^{(0,k)}(M)$ by
\begin{align*}
(D_{S}^{*}\tilde{h})(X_{1}, \ldots, X_{k-1})=&-\sum_{i}(\nabla_{E_{i}}\tilde{h})(E_{i}, X_{1}, \ldots, X_{k-1}).
\end{align*}

We can now define the second order differential operator $\Delta_{\mathrm{Sym}}$ acting on $\Gamma S^{(0,k)}(M)$ via the following Hodge-type decomposition 
\[
\Delta_{\mathrm{Sym}}\doteq D_{S}^{*}D_{S}-D_{S}D_{S}^{*}
\]
 By \cite{Sampson} (see also \cite[Appendix B]{IRV-HO} for a proof) we have that
\begin{equation}\label{f: Samps}
\ \Delta_{\mathrm{Sym}}=\Delta_{B} -\mathfrak{Ric},
\end{equation}
i.e. $\Delta_{\mathrm{Sym}}$ is a Lichnerowicz Laplacian (with  the choice $c=-1$).
\medskip

Exploiting \eqref{f: Samps} we readily deduce the validity of the differential identity
\begin{equation}\label{IneqDeltaSym0}
\frac{1}{2}\Delta \left|T^{S}\right|^2 = -\langle\Delta_{\mathrm{Sym}}T^{S}, T^{S}\rangle-\langle\mathfrak{Ric}(T^{S}), T^{S}\rangle+|\nabla T^{S}|^2,\qquad\forall\,T\in \Gamma T^{(0,k)}(M) .
\end{equation}

\begin{remark}
	\rm{Notice that a totally symmetric $1$-tensor $\omega$ is also a skew-symmetric one-form. In this case, $\Delta_{\mathrm{Sym}}\omega=2\Delta_{B}\omega-\Delta_{H}\omega$, so that \eqref{f: Weitz} and \eqref{f: Samps} are equivalent for $1$-tensors, as it has to be. However, when deducing the Bochner-type formula, the Weitzenb\"ock curvature term appears with a different sign. }
\end{remark}

Let us now move to the proof of our density result.
\begin{proof}[Proof of Theorem \ref{teo_higher}] In our assumptions, we know by \cite[Corollary 5.2]{IRV-HO} that there exists a sequence of cut-off functions $\left\{\chi_{n}\right\}\subset C_{c}^{\infty}(M)$, and a constant $C>0$ independent of $n$ such that, 
\begin{align}
 &\chi_{n}=1 \quad\textrm{on}\quad B_{R_n}(o), \qquad \text{with } \, R_n \doteq C_{H}^{-1}(n-2) \nonumber\\
&|\nabla^{j}\chi_{n}|\leq C\lambda^{-k+j}(r), \quad j=1,\ldots,k-1; \label{proprie_chin}\\
&|\Delta\nabla^{k-2}\chi_{n}|\leq C,\nonumber
\end{align}
where $r$ is the distance from $o$. These cut-off functions were called in \cite{IRV-HO} $k$-th order rough Laplacian cut-offs. It is important to note that the fact that we are asking only for a control on the trace of the $k$-th covariant derivative of the cut-offs (which suffices for our scope) reflects on the weakness of the assumptions we are asking for. Indeed, we are demanding a control on the curvature up to a smaller order than usual (case $p>2$).
\smallskip

Since smooth functions are dense in $W^{k,p}(M)$, to prove the density result it is sufficient to consider $f\in C^{\infty}(M)\cap W^{k,p}(M)$; see for instance \cite{GGP-AccFinn}. We want to prove that $\chi_{n}f$ converges to $f$ in $W^{k,p}(M)$. The lower order terms 
\[
\ \int_{M}|\nabla^{j}(\chi_{n}f)-\nabla^{j}f|^p, \qquad 0 \le j \le k-1
\]
are easily seen to vanish as $n \to \infty$ by using the Cauchy-Schwarz inequality, Lebesgue convergence theorem and the properties of the cut-off functions. Regarding the $k$-th order derivative, we write 
\begin{align*}
\int_{M}|\nabla^{k}(\chi_{n}f)-\nabla^{k}f|^p =&\int_{M}\left|\left[\sum_{i=0}^{k}{{k}\choose{i}} \nabla^{k-i}\chi_{n}\otimes\nabla^{i}f\right]-\nabla^{k}f\right|^p \\
\leq&C\,\int_{M}(1-\chi_{n})^p|\nabla^{k}f|^p+\sum_{i=0}^{k-1}{{k}\choose{i}}\int_{M}|\nabla^{k-i}\chi_{n}|^p|\nabla^{i} f|^p .
\end{align*}
Taking into account the properties of the cut-off functions, all of the addenda vanish as $n \to \infty$ with the possible exception of the one corresponding to $i=0$. Applying H\"older inequality as in \eqref{f: holder} we deduce that, in order to conclude, it is enough to show that
\[
\ \int_{M}|f|^{p}|\nabla^{k}\chi_{n}|^2\to 0
\]
as $n\to\infty$. Define $h_n=\nabla^{k-1}\chi_{n}$ and its simmetrization $h_n^S$. Because of \eqref{IneqDeltaSym0}, 
\begin{align*}
\frac{1}{2}\mathrm{div}\left(|f|^{p}\nabla\left|h_n^{S}\right|^2\right)= &|f|^{p}\left[-\langle\Delta_{\mathrm{Sym}}h_n^{S}, h_n^{S}\rangle-\langle\mathfrak{Ric}(h_n^{S}), h_n^{S}\rangle+|\nabla h_n^{S}|^2\right]\\
&+ \frac{1}{2}\langle\nabla(|f|^p), \nabla (|h_n^{S}|^2)\rangle,\nonumber
\end{align*}
thus integrating and using \cite[Corollary 9.3.4]{petersen} to control the curvature term we get 
\begin{align}\label{IneqDeltaSym}
\int_{M}\langle\Delta_{\mathrm{Sym}}h_n^{S}, |f|^{p}h_n^{S}\rangle  \leq &-\int_{M}|f|^{p}\langle\mathfrak{Ric}(h_n^{S}), h_n^{S}\rangle +\int_{M}|f|^{p}|\nabla h_n^{S}|^2 \\
& + \frac{1}{2}\langle\nabla(|f|^p), \nabla (|h_n^{S}|^2)\rangle \nonumber\\
\leq & (-\alpha)C\int_{M}|f|^{p}|h_n^{S}|^2 +\int_{M}|f|^{p}|\nabla h_n^{S}|^2 \nonumber\\
& + \frac{1}{2}\langle\nabla(|f|^p), \nabla (|h_n^{S}|^2)\rangle.\nonumber
\end{align}
Suppose first that $p\in(1,2]$. By Young's inequality, the regularity Lemma \ref{lem_reg} and the properties  of $h_n$, 
\begin{align}\label{IneqDeltaSym_2}
\int_{M}\langle\Delta_{\mathrm{Sym}}h_n^{S}, |f|^{p}h_n^{S}\rangle  \leq &  (-\alpha)C\int_{M \setminus B_{R_n}(o)}|f|^{p}|h_n^{S}|^2 +\int_{M}|f|^{p}|\nabla h_n^{S}|^2 \\
& +\eta \int_M |f|^p|\nabla|h_{n}^S||^2 + \frac{1}{\eta} \int_{M\setminus B_{R_n}(o)} |\nabla |f|^{\frac p2}|^2\nonumber,
\end{align}
for any $\eta>0$. Notice also that
\begin{equation}\label{ConsCS}
|h_n^S| \le |h_n| = |\nabla^{k-1}\chi_{n}|\leq C\lambda^{-1}(R_n). 
\end{equation}
By the dominated convergence theorem, the fact that $f\in W^{k,p}(M)$ and Lemma \ref{lem_reg}, the first and fourth term in the RHS of \eqref{IneqDeltaSym_2} vanish as $n \to \infty$, so using Kato's inequality $|\nabla|h_n^S||\leq|\nabla h_n^S|$ we obtain
\begin{align}\label{IneqDeltaSym2}
\limsup_{n\to \infty}\left[\int_{M}\langle\Delta_{\mathrm{Sym}}h_n^{S}, |f|^{p}h_n^{S}\rangle  -(1+\eta) \int_M |f|^p|\nabla h_n^S|^2\right]\leq 0.
\end{align}
Define 
	\[\mathscr{A}_n= \int_{M}\langle D_{S}^{*}D_{S}h_n^S, |f|^{p}h_n^S\rangle, \qquad \mathscr{B}_n=\int_{M}\langle D_{S}D_{S}^{*} h_n^S,|f|^{p}h_n^S\rangle,
	\]
so that \eqref{IneqDeltaSym2} becomes
\begin{equation}\label{AB}
\limsup_{n\to \infty}\left[ \mathscr{A}_n - \mathscr{B}_n -(1+\eta) \int_M |f|^p|\nabla h_n^S|^2\right]\leq 0.
\end{equation}
\medskip

By Young's inequality and using \eqref{ConsCS}, for $\delta>0$ we can estimate $\mathscr{A}_n$ as follows:
\begin{align}\label{A}
\mathscr{A}_n=&\int_{M} \langle D_S h_n^S, D_S(|f|^p h_n^S)\rangle  \\
=&\int_{M}|f|^{p}|D_{S}h_n^S|^2 + k\int_M \langle D_{S}h_n^S, 2|f|^{\frac p2}s_{k}\left(d|f|^{\frac p2}\otimes h_n^S\right)\rangle  \nonumber\\
\geq& (1-\delta)\int_{M}|f|^{p}|D_{S}h_n^S|^2\, -\frac{k^2}{\delta}\int_{M}|\nabla (|f|^{\frac p2})|^2\,|h_n^S|^2 ,\nonumber \\
= & (1-\delta)\int_{M}|f|^{p}|D_{S}h_n^S|^2 + o_n(1) \qquad \text{as } \, n \to \infty,\nonumber
\end{align}
where the last line follows by the regularity Lemma and since $h_n^S$ is bounded and supported away from $B_{R_n}(o)$.
\medskip

Regarding the term $\mathscr{B}_n$, H\"older inequality gives
\begin{align}\label{B}
\mathscr{B}_n=&\int_{M}\langle D_{S}^{*}h_n^S, D_{S}^{*}(|f|^{p}h_n^S) \rangle  \\
=& \int_{M}\left[|f|^p|D_{S}^{*} h_n^S|^2 -\langle i_{\nabla (|f|^{p})}h_n^S, D_{S}^{*}h_n^S\rangle\right] \nonumber\\
\leq &\int_{M}|f|^{p}|D_{S}^{*}h_n^S|^2 +\int_{M}|\nabla( |f|^{p})||D_{S}^{*}h_{n}^{S}|\,|h_n^S|\nonumber\\
\leq&\int_{M}|f|^{p}|D_{S}^{*}h_{n}^{S}|^2+p\left(\int_{M}|f|^{p}|D_{S}^{*}h_{n}^{S}|^{\frac{p}{p-1}}\right)^{\frac{p-1}{p}}\left(\int_{M}|\nabla|f||^{p}|h_{n}^{S}|^p\right)^{\frac{1}{p}}\nonumber
\end{align}
By the Ricci identities, a computation (see \cite[pp. 31]{IRV-HO}) shows that 
	\begin{align*}
	|D_{S}^{*}h_n^S|^2 & = |D_{S}^{*} (\nabla^{k-1}\chi_n)^S|^2 \\
	& \le C \Big( |\Delta\nabla^{k-2}\chi_{n}|^2+|\mathrm{Riem}|^2|\nabla^{k-2}\chi_{n}|^2+\ldots+|\nabla^{k-3}\mathrm{Riem}|^2|\nabla\chi_{n}|^2\Big) \\
	& \le \left\{ \begin{array}{ll}
	C' & \text{on } M\setminus B_{R_n}(o) \\
	0 & \text{otherwise}
	\end{array}\right. \qquad \text{by our decay assumptions on ${\rm Riem}$ and by \eqref{proprie_chin}.}
	\end{align*}
Hence, \eqref{B} and $f \in W^{k,p}(M)$ imply that $\disp \limsup_{n \to \infty} \mathscr{B}_n \le 0$. Note that these estimates for $\mathscr{B}_n$ also hold for $p=1$, and indeed the fourth line of \eqref{B} is unnecessary in such case.
\medskip

 Inserting \eqref{A} and \eqref{B} into \eqref{AB} gives
\begin{equation}\label{deltaeta}
\limsup_{n\to \infty}\int_{M}|f|^{p}\left[(1-\delta)|D_{S}h_n^S|^2-(1+\eta)|\nabla h_n^S|^2\right] \le 0.
\end{equation}
Moreover, by the same reasoning as above and by Young's inequality (see \cite[pp.32-33]{IRV-HO}),
\begin{align}
|\nabla h_{n}^S|^2 & = |\frac{1}{(k-1)!}\nabla s_{k-1}(\nabla^{k-1}\chi_{n})|^2\label{eq_tensors2}\\
& \leq (1+\varepsilon C_{1,k})|\nabla h_{n}|^2+\frac{C_{1,k}}{\varepsilon}\Big( |\mathrm{Riem}|^2|\nabla^{k-2}\chi_{n}|^2+\ldots+|\nabla^{k-3}\mathrm{Riem}|^2|\nabla\chi_{n}|^2\Big),\nonumber \\
|D_{S} h_{n}^S|^2 & = k^2 | s_{k}(\nabla s_{k-1}(\nabla^{k-1}\chi_{n}))|^2=k^{2}|s_{k}(\nabla^{k}\chi_{n})|^2\label{eq_tensors3}\\
& \geq (k^2-\varepsilon C_{2,k})|\nabla h_{n}|^2-\frac{C_{2,k}}{\varepsilon}\Big( |\mathrm{Riem}|^2|\nabla^{k-2}\chi_{n}|^2+\ldots+|\nabla^{k-3}\mathrm{Riem}|^2|\nabla\chi_{n}|^2\Big),\nonumber
\end{align}
for any $\varepsilon>0$, and some constants $C_{1,k},C_{2,k}$. 
Using \eqref{eq_tensors2}, \eqref{eq_tensors3},  the decay assumptions on $\mathrm{Riem}$ and $f\in L^{p}(M)$, we get that
\[
\ \limsup_{n\to\infty}\int_{M}|f|^{p}\left[(1-\delta)(k^2-\varepsilon C_{2,k})-(1+\eta)(1+\varepsilon C_{1,k})\right]|\nabla h_{n}|^2\leq 0
\]
 Hence, we can choose $\delta,\eta, \varepsilon$ small enough such that  $(1-\delta)(k^2-\varepsilon C_{2,k})-(1+\eta)(1+\varepsilon C_{1,k})>0$, which leads to
\begin{equation*}\label{deltaeta_2}
\int_{M}|f|^{p}|\nabla h_n|^2 \to 0 \qquad \text{as } \, n \to \infty,
\end{equation*}
thus concluding the proof for $p\in(1,2]$. 
\medskip

We suppose now that $p=1$. We first note that 
\begin{equation}\label{claimp1}
\lim_{n\to \infty} -\frac{1}{2}\int_{M}\langle\nabla|f|, \nabla (|h_n^{S}|^2)\rangle\,  =0.
\end{equation}
Indeed, by Lebesgue convergence theorem, 
\[
-\frac{1}{2}\int_{M}\langle\nabla|f|, \nabla (|h_n^{S}|^2)\rangle\,  =\lim_{\eps\to 0}-\frac{1}{2}\int_{M}\langle\nabla((f^2+\varepsilon)^{1/2}), \nabla (|h_n^{S}|^2)\rangle\,  
\] 
So performing the same computations as in \eqref{f: p1 1}, \eqref{f: p1 2} and \eqref{f: p1 chiR} we obtain 
\begin{equation*}\label{121314}
-\frac{1}{2}\int_{M}\langle\nabla((f^2+\varepsilon)^{1/2}), \nabla (|h_n^{S}|^2)\rangle\,  \leq \|h_n^{S}\|^2_\infty\int_{M}|\Delta f|.   
\end{equation*}
Since the RHS above vanishes as $n\to\infty$ because of \eqref{ConsCS}, this proves the claimed identity \eqref{claimp1}. From \eqref{IneqDeltaSym} we therefore deduce
\begin{equation}\label{DeltaSym3}
\limsup_{n\to \infty}\left[\mathscr{A}_n - \mathscr{B}_n -\int_M |f||\nabla h_n^S|^2\right]\leq 0.
\end{equation}
As the estimate for $\mathscr{B}_n$ holds also for $p=1$, we only have to deal with $\mathscr{A}_n$:
\begin{align*}
\mathscr{A}_n=&\int_{M}\langle D_{S}h_{n}^{S}, D_{S}(|f|h_{n}^{S})\rangle\\
=&\int_{M}|f||D_{S}h_{n}^{S}|^2+k\int_{M}\langle D_{S}h_{n}^{S},s_{k}(d|f|\otimes h_{n}^{S})\rangle.
\end{align*}
By Lebesgue convergence theorem,
\[
\int_{M}\langle D_{S}h_{n}^{S}, s_{k}(d|f|\otimes h_{n}^{S})\rangle =\lim_{\varepsilon\to 0}\int_{M}\langle D_{S}h_{n}^{S}, s_{k}(d(\sqrt{f^{2}+\varepsilon})\otimes h_{n}^{S})\rangle,
\]
hence we compute
\begin{align*}
\left| \int_{M}\langle D_{S}h_{n}^{S},s_{k}(d|f|\otimes h_{n}^{S})\rangle \right| =& \left|\lim_{\varepsilon\to0}\int_{M}\langle h_{n}^{S}, D_{S}^{*}(s_{k}(d(\sqrt{f^{2}+\varepsilon})\otimes h_{n}^{S}))\rangle \right|\\
=& \left|\lim_{\varepsilon\to 0}\left[\int_{M}\frac{f\Delta f}{\sqrt{f^{2}+\varepsilon}}|h_{n}^{S}|^2+\frac{f}{2\sqrt{f^{2}+\varepsilon}}\langle\nabla f, \nabla |h_{n}^{S}|^2\rangle+\frac{\varepsilon}{(f^{2}+\varepsilon)^{3/2}}|\nabla f|^2|h_{n}^{S}|^2\right]\right|\\
\leq& 3 \left\|h_{n}^{S}\right\|^2_{\infty}\int_{M}|\Delta f|,
\end{align*}
where for the last inequality we reasoned again as in \eqref{f: p1 1}, \eqref{f: p1 2} and \eqref{f: p1 chiR}. Summarizing, 
\[
\mathscr{A}_n =\int_{M}|f||D_{S}h_{n}^{S}|^2+o_{n}(1) \qquad \text{as } \, n \to \infty,
\]
and the proof can be concluded as in the case $p>1$.
\end{proof}


\section{Non-density when $p > 2$: a counterexample with curvature $\Sec \ge - 1$}
To begin with, we construct a suitable complete, convex hypersurface $(M,g_0) \hookrightarrow \HH^{n+1}$ of finite volume and with two ends. Let us consider cartesian coordinates $(\mathbf x, z)=(x_1,\dots,x_n,z)$ on $\R^{n+1}$. Let $\mathbb B_1=\{|\mathbf x|^2 + z^2<1\}$ be the unit ball centered at the origin. Let $h$ be the hyperbolic metric on $\mathbb B_1$ induced by the Beltrami-Klein projective model, i.e. 
\[
h= \frac{\|d\mathbf y\|^2}{1-\|\mathbf y\|^2}+\frac{(\mathbf y\cdot d\mathbf y)^2}{(1-\|\mathbf y\|^2)^2}, 
\] 
where $\mathbf y\in \mathbb B_1$ and $\|\cdot\|$ is the standard Euclidean norm of $\R^{n+1}$. 
Define the noncompact hypersurface $M$ by 
\[
M= \Big\{|\mathbf x|=-\sqrt 3 + \sqrt{4-z^2}\ :\ z\in(-1,1)\Big\}\subset \mathbb B_1,
\] 
and let $g_0$ be the metric on $M$ induced by $h$. Note that $M$ is the boundary of a domain which is strictly convex in $\R^{n+1}$, hence also in $(\mathbb B_1,h)$ since the Beltrami-Klein model is projective. Thus $\Sec_{g_0} > -1$ by Gauss equations. Furthermore, $M$ is invariant by reflection with respect to the plane $z=0$, and $M \cap \{z \ge 0\}$ can be written as the graph of the strictly concave function 
	\begin{equation}\label{def_f_graph}
	f : D \doteq \overline{B_{2-\sqrt{3}}^{\R^n}(0)} \setminus \{0\} \to [0,\infty), \qquad f(\mathbf{x}) = \sqrt{1 - |\mathbf{x}|^2 - 2\sqrt{3} |\mathbf{x}|}, 
	\end{equation}
where $B_{2-\sqrt{3}}^{\R^n}(0)$ is the Euclidean ball of radius $2-\sqrt{3}$ in $\{z=0\}$. Hereafter, we will shortly say that $M$ is the bigraph of $f$. Denote with $\tilde f = \mathrm{id} \times f$ the graph map. Note that $(M,g_0)$ lies in the interior region of the double cone  
\begin{equation}\label{def_K}
K=\left\{|\mathbf x|=\frac{1-|z|}{\sqrt 3}\ :\ z\in(-1,1)\right\}\stackrel{i}{\hookrightarrow} (\mathbb B_1,h), 
\end{equation}
and that $K$ has finite volume. This can be easily proved by a direct computation, for instance by noticing that each of the two cones forming $K$ is isometric to the half cylinder $\{|\mathbf x|=1, z>1\}$ in the Poincar\'e half-space model. Since the orthogonal projection on a convex set of $\mathbb H^{n+1}$ is distance decreasing by the hyperbolic Buseman-Feller theorem \cite[II.2.4]{BH}, we deduce that $(M,g_0)$ has finite volume. We fix 
	\[
\vm \Subset \ \ \text{bigraph of $f$ over } \ \overline{B^{\R^n}_{2-\sqrt{3}}(0)}\setminus B^{\R^n}_{\frac 18}(0)
	\]
whose closure is diffeomorphic to a closed ball (in particular, $\vm$ does not disconnect $M$) and we define 
	\[
	U_0 \doteq \emptyset, \qquad U_j \doteq \ \text{bigraph of $f$ over } \ B_{2-\sqrt{3}}^{\R^{n}}(0)\setminus \overline{B^{\R^n}_{\frac 1{j+8}}(0)} \ \ \ \text{ for } \, j \ge 1.
	\]	
	
Roughly speaking, $M$ looks like an American football in vertical position with respect to $\{z=0\}$, and $U_j$ corresponds to the open set obtained by removing an upper and a lower cap centered at the two vertices. 

We begin by constructing, for fixed $j$, a sequence of smooth metrics $\{\sigma_{j,k}\}_{k=0}^\infty$ on $M$ having $k$ ``approximated spikes" in $U_j \setminus U_{j-1}$ and converging, as $k \to \infty$, to an Alexandrov metric that has a dense set of sharp points on $U_j \setminus U_{j-1}$. This is the content of the next Section. Before we get going, let us recall the notion of sharp singular point, and some basic facts of Alexandrov (more generally, $\RCD$) spaces that will be useful later on. The theory of metric measure spaces $(X,\di, {\sf m})$ (${\sf m}$ a Radon measure on $X$) that lie in $\RCD(K,n)$ hugely developed in the past 20 years, and for an informative account, with a detailed set of references, we recommend  \cite{ambrosio_survey}. Here, we just point out that $\RCD(K,n)$ contains all Alexandrov spaces with dimension $n$ and curvature bounded from below by $K/(n-1)$, with ${\sf m}$ the $n$-dimensional Hausdorff measure, as well as the pointed measured Gromov-Hausdorff (mGH) limits of smooth manifolds $(M_i,g_i, o_i)$ with $\Ricc \ge K$, endowed with their Riemannian measure ${\sf m}_i$ and reference points $o_i$. For $X \in \RCD(K,n)$, the Sobolev spaces $W^{1,p}(X)$ can be defined for $p \in (1,\infty)$, and $W^{1,2}(X)$ is Hilbert. Given $(X,\di, {\sf m}) \in \RCD(K,n)$ and $x_0 \in X$, the density
	\[
	\vartheta(x_0) \doteq \lim_{r \to 0} \frac{\me(B_r(x_0))}{r^n} \in (0, \infty]
	\]
does exist. A tangent cone at $x_0$ is, by definition, the mGH limit of some sequence of rescalings
	\[
	\left(X, \frac{\di}{\lambda_i}, \frac{\me}{\lambda_i^n}, x_0\right) \qquad \text{where } \, \lambda_i \to 0^+,
	\]
and the set of tangent cones is closed under mGH convergence pointed at $x_0$. Under the non-collapsing condition $\vartheta(x_0) < \infty$, every tangent cone at $x_0$ is a metric cone $C(Z)$ over a cross section $Z \in \RCD(n-1,n)$ with diameter $\le \pi$, that is, it can be written as $[0,\infty) \times Z$ with distance
	\[
	\di_{C(Z)} \big( (t,x), (s,y)\big) = \sqrt{t^2 + s^2 - 2ts \cos\big( \di_Z(x,y)\big)}.
	\]
The section is unique for Alexandrov spaces, but this may not be the case in general. Following \cite{dePNZ}, we say that $x_0 \in X$ is \emph{sharp} if $\vartheta(x_0) < \infty$ and the cross section of any tangent cone at $x_0$ has diameter $< \pi$.

\subsection{Construction of the spike metrics $\sigma_{j,k}$}\label{AppendixA}

It is well-known that there exist manifolds $(M,g_k)$ with $\Sec_{g_k} \ge 0$ that converge to an Alexandrov space having a dense set of sharp singular points, \cite{OtsuShioya}. In the next Lemma we will need to localize such a construction, namely, to approximate the singular points in $U_{j}\setminus U_{j-1}$ without modifying the metric $g_{j}$ outside. To this end, we adapt the construction introduced in \cite{MV} to a hyperbolic background. As we shall need more information on the sequence of approximating metrics, the proof of the next result will be done in full detail.

\begin{lemma}\label{l: approximation}
		For $j\geq 1$, there exists a sequence of smooth metrics $\{\sigma_{j,k}\}_{k \in \mathbb{N}}$ on $M$ such that 
	\begin{align}
	& \text{$\sigma_{j,k} = g_0$ outside of a compact subset of $U_{j}\setminus U_{j-1}$ (depending on $k$),} \\
	& \Sec_{\sigma_{j,k}} > -1 \qquad \text{on } \, M,  \\ 
	& \forall \, k : \NN_{>0} \to \NN, \quad \forall \, S \subset M \ \text{Borel}, \qquad \sum_{j=1}^\infty \vol_{\sigma_{j,k(j)}}\big(S \cap (U_j \setminus U_{j-1})\big) \le \vol_{i^\star h}(K)< \infty \label{proprie_basic} \\
	& \exists \, C_j>1 \ \text{such that} \quad C_j^{-1} \di_{g_0}(x,y) \le \di_{\sigma_{j,k}}(x,y) \le C_j \di_{g_0}(x,y) \qquad \forall \, k \in \NN \cup \{0\}, \ x,y \in M. \label{eq_bilipschitz}
	\end{align}
Moreover, $(M, d_{\sigma_{j,k}}, o) \to M_{j,\infty} \doteq (M,\di_{j,\infty}, o)$ as $k\to\infty$ in the Gromov-Hausdorff sense, for some $n$-dimensional Alexandrov space $M_{j,\infty}$ biLipschitz homeomorphic to $M$, with curvature greater than or equal to $-1$, volume $\mathcal H^n(M_{j,\infty})\le \vol_{i^\star h}(K)< \infty$, and a dense set of sharp singular points in $U_j \setminus U_{j-1}$. 
	\end{lemma}
	
\begin{proof}
Define 
	\[
	D_0 = \emptyset, \qquad D_j = \overline{B_{2-\sqrt{3}}^{\R^{n}}(0)} \setminus \overline{B^{\R^n}_{\frac 1{j+8}}(0)},
	\]
so $U_j = \tilde{f}(D_j) \cup \widetilde{(-f)}(D_j)$ is the bigraph of $f$ over $D_j$. Note that $f$ satisfies $f(\mathbf x) < 1 -  \sqrt 3 |\mathbf x|$, since the graph of this latter function coincides with $K$ on $\{z \ge 0\}$. Let $\{\mathbf{y}_m\} \in D_j \setminus \overline{D_{j-1}}$ be a dense sequence. We claim that 
\begin{quote}{\it
	there exists a sequence of smooth strictly concave functions $f_{j,k}:D\to \R$, $k\geq 1$, such that
\begin{enumerate}
	\item[(i)] $f(\mathbf x)\le f_{j,k}(\mathbf x) < 1- \sqrt 3 |\mathbf x|$ on $D$;
	\item[(ii)] $f_{j,k}$ converges uniformly, as $k \to \infty$, to a concave function $f_{j,\infty}$, and the graph of $f_{j,\infty}$ has sharp conical singularities at any $f_{j,\infty}(\mathbf y_m)$;
	\item[(iii)] $\{{\bf x} : f_{j,k}({\bf x}) \neq f(\bf x)\}$ is compactly contained in $D_j \setminus \overline{D_{j-1}}$. 
\end{enumerate}}
\end{quote}
Given the claim, let $(M_{j,k},\sigma_{j,k})$ be the bigraph of $f_{j,k}$ with the induced metric. Property (ii) implies the Hausdorff convergence of $M_{j,k}$ to the bigraph $M_{j,\infty}$ of $f_{j,\infty}$ with the induced intrinsic metric $\di_{j,\infty}$ and it is known that the concavity of $f_{j,k}$ guarantees the pointed Gromov-Hausdorff convergence $(M_{j,k},\di_{j,k},o_k) \to (M_{j,\infty}, \di_{j,\infty},o_\infty)$, with $o_k$ being the image of any fixed point in $D_1$. Using again the concavity of $f_{j,k}$, Gauss' equation implies that $M_{j,k}$ has sectional curvature bounded from below by $-1$, and $(M,\di_{j,\infty})$ is an Alexandrov space of curvature lower bounded by $-1$ by Buyalo's theorem, \cite{buyalo}. Next, for $0 \le k \le \infty$, identify $M$ with $M_{j,k}$ topologically via the map $\tilde f_{j,k} \circ \tilde f^{-1}$, and still denote with $\sigma_{j,k}$ the pulled-back metric on $M$. Note that $\{g_{j,k} \neq g_0\}$ is compactly contained in $U_j \setminus \overline{U_{j-1}}$. The uniform convergence together with the concavity of $f_{j,k}$ on $D$ guarantee that $\{f_{j,k}\}_k$ are uniformly Lipschitz on $D_j$, hence on the entire $D$ by (iii). In particular, up to identifying the manifolds by means of $\tilde f_{j,k} \circ \tilde f^{-1}$, \eqref{eq_bilipschitz} holds. To conclude, for a given $k : \NN_{>0} \to \NN$ we consider the concave function $f_\infty$ that equals $f_{j,k(j)}$ on $D_j \setminus D_{j-1}$. By the above construction, the bigraph $(M, g_\infty)$ of $f_\infty$ is the boundary of a convex set in $(\mathbb{B}_1, h)$ contained in $K$, so by the hyperbolic Busemann-Feller theorem the nearest point projection from $K$ to $(M,g_\infty)$ is distance decreasing. In particular, for every Borel set $S \subset M$ it holds $\vol_{g_\infty}(S) \le \vol_{i^\star h}(K)$, proving \eqref{proprie_basic}. 

It remains to prove the claim. In \cite{MV} it is presented a general procedure to construct a sequence of metrics on a bounded set of a Riemannian manifold which  Gromov-Hausdorff converges to an Alexandrov space with a sharp conical singularity at each point of a countable set. For completeness, we reproduce here the construction in our setting. Consider $g : \R^n \to \R$ such that
\begin{equation*}
\begin{cases}
g({\bf x}) = 1- |\mathbf x| - |\mathbf x|^2  & \ \ \text{for } \, \mathbf x \in B_{1/2}^{\mathbb{R}^{n}}\\
g \in C^\infty(B_1^{\mathbb{R}^{n}} \setminus \lbrace 0\rbrace) \\
\operatorname{supp} g \subseteq B_1^{\mathbb{R}^{n}}\\
g \ge 0.
\end{cases}
\end{equation*}
Then, for $\varepsilon> 0$ and $\mathbf y \in \R^n$ we define $g_{\varepsilon, \mathbf y }: \R^n \to \R$ as
\begin{equation*}
g_{\varepsilon, \mathbf y }(x) \doteq g\left(\frac{\mathbf x-\mathbf y}{\varepsilon} \right),
\end{equation*}
so that $g_{\varepsilon, \mathbf y }$ is smooth outside $\mathbf y$, non-positive and strictly concave on $B_{\varepsilon/2}^{\mathbb{R}^{n}}(\mathbf y)$.
Let 
	\[
	0<\varepsilon_1 < \operatorname{dist}_{\R^n}(\mathbf y_1,\partial (D_j \setminus D_{j-1}))
	\]
and define
\begin{equation*}
\phi_1(\mathbf x) \doteq f(\mathbf x) + \eta_1 g_{\varepsilon_1, \mathbf y_1}(\mathbf x),
\end{equation*}
with $\eta_1>0$ small enough so that $\phi_1$ is strictly concave and $\phi_1(\mathbf x)<1-\sqrt 3 \mathbf x$ on $D$.
Observe also that $\phi_1$ is smooth on $D \setminus \lbrace \mathbf y_1 \rbrace$ and its graph has a sharp singular point at $\phi_{1}(\mathbf  y_1)$. 

Recursively, let $0<\varepsilon_k < \operatorname{dist}_{\R^n}(\mathbf y_k, \partial (D_j \setminus D_{j-1}) \cup \{ {\bf y}_1, \ldots, {\bf y}_{k-1}\})$ and define
\begin{equation}
\label{eq: f_k^j}
\phi_k(\mathbf x) \doteq \phi_{k-1}(\mathbf x) + \eta_k g_{\varepsilon_k,\mathbf  y_k}(\mathbf x).
\end{equation}
The function $\phi_k$ is smooth on $D \setminus \lbrace \mathbf y_1, \ldots, \mathbf y_k \rbrace$, strictly concave and satisfies $\phi_k(\mathbf x)>\sqrt 3 \mathbf x-1$ provided that $\eta_k$ is small enough. Moreover, the graph of $\phi_k$ has sharp singularities at $\phi_k(\mathbf y_1), \ldots, \phi_k(\mathbf y_k)$.
Furthermore, if $\eta_k$ are such that $\sum_k \eta_k$ converges, then $\phi_k$ converges uniformly to some $\phi_\infty=:f_{j,\infty}$ whose graph is convex, has sharp singularities at $\lbrace \tilde{\phi}_{\infty}(\mathbf y_m)\rbrace_{m = 1}^\infty$, coincides with the graph of $f$ outside of $D_j \setminus D_{j-1}$ and is contained in the double cone $K$. The sharpness of the singularity at each $\tilde{\phi}_{\infty}(\mathbf y_m)$ can be directly checked, making use of the fact that points of an Alexandrov space have a unique tangent cone.

To define the smooth functions $f_{j,k}:D\to\mathbb{R}$ approximating $f_{j,\infty}$, recall that $f_{j,\infty}= f+\sum_{k=1}^\infty \eta_k g_{\varepsilon_k,\mathbf y_k}$. By a diagonal argument, it is enough to show that each $g_{\varepsilon_k,\mathbf y_k}$ can be uniformly approximated by smooth functions which coincide with $g_{\varepsilon_k,\mathbf y_k}$ outside $B_{\varepsilon_k/2}^{\mathbb{R}^{n}}(\mathbf y_k)$.
For $0 < \delta < \varepsilon_k/2$, let $g_{\varepsilon_k,\mathbf y_k,\delta}$ be a smooth function that is  strictly concave on $B_{\varepsilon_k/2}^{\mathbb{R}^{n}}(\mathbf{y}_k)$ and coincides with $g_{\varepsilon_k,\mathbf y_k}$ outside of $B_\delta^{\mathbb{R}^{n}}(\mathbf{y}_k)$, see for instance \cite[Theorem 2.1]{ghomi}. As $\delta\to 0$, we have that $g_{\varepsilon_k,\mathbf y_k,\delta}\to g_{\varepsilon_k,\mathbf y_k}$ uniformly. This concludes the proof.
\end{proof}

Let $E_+,E_-$ be the two connected components of $M \setminus U_1$, respectively contained in $\{z>0\}$ and in $\{z < 0\}$, and for each $j$ define 
	\begin{equation}\label{epmj}
	E_{-,j} \doteq E_- \setminus U_j, \qquad E_{+,j} \doteq E_+ \setminus U_j,
	\end{equation}
The metric $g$ on the block $M$ will be constructed from the original metric $g_0$ by prescribing, for each $i \ge 1$, a spike metric $\sigma_{i,k(i)}$ with $k(i)$ approximated spikes on $U_i \setminus U_{i-1}$. The function $k : \NN_{>0} \to \NN$ will be chosen inductively, by identifying, for each $j \ge 1$, $k(j)$ depending on $k(1), \ldots, k(j-1)$. Correspondingly, to each $j$ we shall associate a smooth metric $g_j$ on $M$ that corresponds to the choices of $\sigma_{i,k(i)}$ on $U_i \setminus U_{i-1}$ for $1 \le i \le j$. In particular, $g_j = g_{j-1}$ outside of $U_j \setminus U_{j-1}$. In the following lemma we summarize the properties of the metrics $g_j$ to be proved.
\begin{lemma}\label{l: construction}
There exists a sequence of metrics $\{g_j\}_{j=1}^\infty$ on $M$ with the following properties:
	\begin{align}
	& \big\{ x : \ g_j(x) \neq g_{j-1}(x) \big\} \ \text{ is compactly contained in } \, U_j \setminus \overline{U_{j-1}}, \label{p1} \tag{$\PP 1$}\\[0.1cm]
	& \Sec_{g_j} \ge -1, \label{p2} \tag{$\PP 2$} \\[0.1cm]
	& \forall \, S \subset M \ \text{Borel}, \qquad  \vol_{g_j}(S) \le \vol_{i^\star h}(K)<+\infty, \qquad   \label{p3} \tag{$\PP 3$} \\[0.1cm]
	& \exists \, \bar C_j>1 \ \text{such that} \quad \bar C_j^{-1} \di_{g_0}(x,y) \le \di_{g_j}(x,y) \le \bar C_j \di_{g_0}(x,y) \qquad \forall x,y \in M. \label{p4} \tag{$\PP 4$} 
	\end{align}   
where $K$ is the double cone defined in \eqref{def_K}, and $\di_{g_j}$ is the distance induced by $g_j$. Furthermore, having defined $E_{\pm,j}$ as in \eqref{epmj}, $g_j$ and $g_{j+1}$ satisfy 
	\begin{equation}\label{p5}\tag{$\PP 5$}
	\forall\, \varphi \in C^\infty(M), \  \ \begin{array}{ll} 
	\varphi \le -1 + 2^{-j} & \text{on } \, \partial E_{-,j} \\[0.2cm]
	\varphi \ge 1 - 2^{-j} & \text{on } \, \partial E_{+,j}
	\end{array}
	\quad \Longrightarrow \quad \|\varphi\|_{W^{2,p}(U_{j+1} \setminus \vm, g_{j+1})} > 1.
	\end{equation}
\end{lemma}

\begin{remark}\label{rem_consistency}
\emph{About \eqref{p5}, we shall see below that $g_j$ matches the following stronger property: whenever $\varphi$ satisfies the assumptions of \eqref{p5}, the inequality
	\[
	\|\varphi\|_{W^{2,p}(U_{j+1} \setminus \vm, \bar g)}
> 1
	\]
will hold for \emph{any} choice of $\bar g$ that coincides with $g_j$ on $U_j$ and with a spike metric $\sigma_{j+1,m}$ on $U_{j+1} \setminus U_j$. In particular, \eqref{p5} does not require to have already chosen the integer $k(j+1)$, but holds a-posteriori for every possible choice of it.
}
\end{remark}

\subsection{Proof of Theorem \ref{Counterex}}

Let us see how Lemma \ref{l: construction} allows to conclude the proof of Theorem \ref{Counterex}.
\medskip

Let $g$ be the smooth Riemannian metric on $M$ defined by $g=g_j$ on $U_j$ for $j\geq 0$. It is readily seen by \eqref{p2},\eqref{p3} that $\Sec_{g}\ge -1$ and that $\vol_{g}(M) \le \vol_{i^\star h}(K)<\infty$. Furthermore, referring to the proof of Lemma \ref{l: approximation}, $(M,g)$ can be realized as the bigraph of a concave function that equals $f_{j,k(j)}$ on $D_j \setminus D_{j-1}$. Such a bigraph is properly embedded in $(\mathbb{B}_1,h)$, hence $(M,g)$ is complete. Let us glue $N$ to $M$ along $\vm'$ and $\vm$, by keeping the metric $g$ unchanged outside of $\vm$. For convenience, still denote with $\vm$ the complement of $M \setminus \vm$ inside of $M \sharp N$, and with $g$ the glued metric. Fix a smooth function $F: M \sharp N \to\R$ such that 
	\[
	F \equiv 0 \ \ \text{ on } \, \vm, \qquad F\equiv -1 \ \ \text{ on } \, E_{-,1}, \qquad F\equiv 1 \ \ \text{ on } \, E_{+,1}.  
	\]
Since $(M,g)$ has finite volume, it is clear that $F\in W^{k,p}(M \sharp N)$ for every $k,p$. For each $p>n$, we prove that $F$ cannot be approximated by compactly supported smooth functions in $W^{2,p}(M\sharp N)$, as the statement for higher $k$ is a simple consequence. Suppose by contradiction that there exists a sequence $\{F_i\}_{i=0}^\infty\subset C^\infty_c(M \sharp N)$ such that $\|F-F_i\|_{W^{2,p}(M \sharp N,g)}\to 0$ as $i\to \infty$. In particular, there exists $i$ such that 
\[
\|F-F_i\|_{W^{2,p}(M \setminus \vm,g)}\le 1/2.
\]
Choose $j \ge 1$ so that $F_i$ has support in $U_j$. Then $F-F_i\equiv -1$ on $E_{-,j}$ and $F-F_i\equiv 1$ on  $E_{+,j}$, hence \eqref{p5} enables us to conclude that 
	\[
	\|F-F_i\|_{W^{2,p}(U_{j+1}\setminus \vm,g_{j+1})}> 1.
	\]
However, since $F-F_i$ is constant outside of $U_j$ and $g = g_{j+1}$ on $U_{j+1}$, 
\begin{equation}\label{f: norm lower bound}
\begin{array}{lcl}
\disp \frac{1}{2} & \ge & \disp \|F-F_i\|_{W^{2,p}(M \setminus \vm,g)} \ge \|F-F_i\|_{W^{2,p}(U_{j+1}\setminus \vm,g_{j+1})} > 1,
\end{array}
\end{equation}
contradiction.

\subsection{Proof of Lemma \ref{l: construction}}
	
Suppose that $g_{j-1}$ is constructed. Let $\left\{g_{j,k}\right\}_{k\in\mathbb{N}\cup\{0\}}$ be the sequence of smooth metrics on $M$ being equal to $g_{j-1}$ outside of $U_j \setminus U_{j-1}$ and equal to the spike metric $\sigma_{j,k}$ on $U_j \setminus U_{j-1}$. Then, $g_{j,0} \equiv g_{j-1}$ on $M$ and, denoting with $\di_{j,k}$ the distance induced by $g_{j,k}$, from Lemma \ref{l: approximation} we easily deduce the following properties: 
	\begin{align*}
	& \big\{x \ : \ g_{j,k}(x) \neq g_{j-1}(x)\big\} \quad \text{ is compactly contained in  } \, U_j \setminus U_{j-1}, \\[0.1cm]
	& \Sec_{g_{j,k}} \ge -1 \qquad \text{for each } \, k, \\[0.1cm]
	& \forall \, S \subset M \ \text{Borel,} \qquad \vol_{g_{j,k}}(S) \le \vol_{i^\star h}(K)<\infty.
	\end{align*}   
For each choice of $k(j)$, the metric $g_{j} \doteq g_{j,k(j)}$ therefore satisfies \eqref{p1},\eqref{p2},\eqref{p3}. To prove \eqref{p4} and \eqref{p5}, for any fixed $k,m \in \mathbb{N}$ we define the smooth metric $g_{j,k,m}$ such that
	\[
	g_{j,k,m} = \sigma_{j+1,m} \quad \text{on } \, U_{j+1}\setminus U_j, \qquad g_{j,k,m} = g_{j,k} \quad \text{otherwise}.
	\]
The construction of $g_{j,k,m}$ and (iii) in Lemma \ref{l: approximation} guarantee that there exists a constant $\bar C_j > 1$ such that
	\begin{equation}\label{bili_2}
	\bar C_j^{-1} \di_{g_{j-1}}(x,y) \le \di_{g_{j,k,m}}(x,y) \le \bar C_j \di_{g_{j-1}}(x,y) \qquad \forall x,y \in M, \ k,m \in \NN \cup \{0\}.
	\end{equation}
In particular, independently of the possible choice of $k(j)$, $g_j$ also satisfies \eqref{p4}. Observe that \eqref{bili_2} implies that
	\begin{equation}\label{eq_monotonicity}
	\exists \ \nu_j > 0 \ \text{ such that } \qquad \vol_{g_{j,k,m}}\big( B^{g_{j,k,m}}_1(z) \big) \ge \nu_j \quad \forall \, z \in U_j, \ k,m \in \NN.
	\end{equation}	
As anticipated in Remark \ref{rem_consistency}, we shall prove the following strengthened version of \eqref{p5}:

\begin{quote}
\textbf{Claim 1:} there exists $k(j)$ depending on $j$ such that $g_{j} \doteq g_{j,k(j)}$ satisfies  
\begin{equation}\label{p5'}\tag{$\PP 5'_j$}
\forall\, \varphi \in C^\infty(M), \  \ \begin{array}{ll} 
\varphi \le -1 + 2^{-j} & \text{on } \, \partial E_{-,j} \\[0.2cm]
\varphi \ge 1 - 2^{-j} & \text{on } \, \partial E_{+,j}
\end{array}
\quad \Longrightarrow \quad  \forall\,m,\ \|\varphi\|_{W^{2,p}(U_{j+1} \setminus \vm, g_{j, k(j),m})} > 1.
\end{equation}
\end{quote}

Assume, by contradiction, that \eqref{p5'} does not hold, so that, for $k$ large enough, there exists a sequence $\{\varphi_{j,k}\}$ with $\varphi_{j,k} \in C^\infty(M,g_{j,k})$, and a sequence of integers $\{m_k\}$, such that
	\begin{equation}\label{proprie_contrad}
	\begin{array}{ll} 
	\varphi_{j,k} \le -1 + 2^{-j} & \text{on } \, \partial E_{-,j} \\[0.2cm]
	\varphi_{j,k} \ge 1 - 2^{-j} & \text{on } \, \partial E_{+,j}
	\end{array}
	\qquad \text{but} \qquad  \|\varphi_{j,k}\|_{W^{2,p}(U_{j+1} \setminus \vm,g_{j,k,m_k})} \le 1.	
	\end{equation}
We examine the convergence of the sequence $\left\{\varphi_{j,k}\right\}_k$ on $\overline{U_j} \setminus \overline{\vm}$. 

\begin{quote}
\textbf{Claim 2:} {\it as $k \to \infty$, the sequence $\varphi_{j,k}$ converges locally uniformly on $\overline{U_j} \setminus \overline{\vm}$ to a function $\varphi_j$ that is locally H\"older continuous on $\overline{U_j} \setminus \overline{\vm}$ and locally constant on $U_{j} \setminus U_{j-1}$ (on $U_1 \setminus \vm$, if $j=1$). 
} 
\end{quote} 
We describe how Claim 2 yields to the proof of Claim 1. First, since the convergence is uniform up to the boundary of $U_j$, passing to the limit we obtain 	
	\begin{equation}\label{boundary_varphi}
	\varphi_j \ge 1 \ \ \text{ on } \, \partial E_{+,j}, \qquad \varphi_j \le -1 \ \ \text{ on } \, \partial E_{-,j}. 
	\end{equation}
The argument goes then by induction on $j$. If $j=1$, $U_1 \setminus \vm$ is connected and thus $\varphi_1$ is constant. This contradicts the fact that $\partial E_{+,1} \cup \partial E_{-,1} \subset \partial (U_1 \setminus \vm)$. Having proved Claim 1 for $j=1$, and thus having constructed $g_1$ with property $(\PP 5'_1)$, we examine the case $j>1$. We proceed inductively, that is, we assume to have constructed $g_{j-1}$ in such a way that \eqref{p1},$\ldots$, $(\PP 5'_{j-1})$ hold. If $j>1$, then $U_{j}\setminus U_{j-1}$ has at least two connected components, respectively contained in $E_{+}$ and $E_{-}$. The constancy of $\varphi_j$ on each component, coupled with \eqref{boundary_varphi}, guarantees that 
   \[
\varphi_j \ge 1 \ \ \text{ on } \, \partial E_{+,j-1}, \qquad \varphi_j \le -1 \ \ \text{ on } \, \partial E_{-,j-1}.
	\]
Therefore, for $k$ large enough, 
   \[
\varphi_{j,k} \ge 1 - 2^{-j+1} \ \ \text{ on } \, \partial E_{+,j-1}, \qquad \varphi_{j,k} \le -1 + 2^{-j+1} \ \ \text{ on } \, E_{-,j-1},
	\]
and thus, by $(\PP 5'_{j-1})$, 
	\[
	\|\varphi_{j,k}\|_{W^{2,p}(U_{j}\setminus \vm,g_{j,k})} > 1.
	\]
Concluding, since $g_{j,k,m_k} = g_{j,k}$ on $U_{j}$,
	\[
	1 \ge \|\varphi_{j,k}\|_{W^{2,p}(U_{j+1}\setminus \vm,g_{j,k,m_k})} \ge \|\varphi_{j,k}\|_{W^{2,p}(U_{j}\setminus \vm,g_{j,k})} > 1, 
	\]
contradicting \eqref{proprie_contrad}.\\[0.2cm]	
It remains to prove Claim 2. The argument is inspired by the recent \cite{dePNZ}, where the authors study the behaviour of harmonic functions near sharp points of $\RCD(K,n)$ spaces. Recall that, given a complete metric $\bar g$ on $M$ with $\mathrm{Ric}_{\bar{g}} \ge -(n-1)$, and a geodesic ball $B_R(o)$ centered at some fixed origin $o$, there exist constants $C_{\dou}, C_\dou'$ depending on $n,R$ such that
	\begin{equation}\label{eq_doubling}
	\vol_{\bar{g}}\big(B_{2r}(z)\big) \le C_\dou \vol_{\bar{g}}\big(B_{r}(z)\big) \qquad \forall \, B_{2r}(z) \subset B_R(o)
	\end{equation}
and, for every $0 < r < s$ such that $B_s(z) \subset B_R(o)$,  
	\begin{equation}\label{eq_reversedoubling_0}
	\frac{\vol_{\bar{g}} \big(B_r(z)\big)}{\vol_{\bar{g}}\big(B_s(z)\big)} \ge \frac{V_{-1}(r)}{V_{-1}(s)} \ge C_\dou' \left(\frac{r}{s}\right)^n,
	\end{equation}	
where $V_{-1}(t)$ is the volume of a ball of radius $t$ in the $n$-dimensional hyperbolic space of curvature $-1$. It is a simple consequence of the above two inequalities that there exists $C_\dou'' = C_\dou''(n,R)$ such that
	\begin{equation}\label{eq_reversedoubling}
	\text{for each } \, B'_r \subset B_s \ \text{ geodesic balls in } \, B_R(o), \qquad \frac{\vol_{\bar{g}}(B'_r)}{\vol_{\bar{g}}(B_s)} \ge C_\dou''(n,R) \left(\frac{r}{s}\right)^n,  
	\end{equation}
where now $B'_r,B_s$ may not be concentric. On the other hand, Buser's isoperimetric inequality \cite{buser} (see \cite[Th. 5.6.5]{saloff} or \cite[Thm. 1.4.1]{koreschoen} for alternative proofs) guarantees the existence, for each $p \in [1,\infty)$, of a constant $\Po_p= \Po_p(n,p,R)$ such that
\begin{equation}\label{buser}
\begin{array}{lcl}
\disp \left\{ \fint_{B_r(x)} |\psi - \bar \psi_{B_r(x)}|^p \right\}^{\frac{1}{p}} & \le & \disp r \Po_p \left\{\fint_{B_r(x)} |\nabla \psi|^p \right\}^{\frac{1}{p}} \qquad \forall \, \psi \in \lip(B_{R}(o)),
\end{array}
\end{equation}
where $\bar \psi_{B_r(x)}$ is the mean value of $\psi$ on $B_r(x)$. 

Because of Lemma \ref{l: approximation}, up to subsequences $(M, g_{j,k,m_k},o) \to M_{j,\infty} \doteq (M, \di_{j,\infty},o)$ as $k \to 0$ in the Gromov-Hausdorff sense, where $M_\infty$ is an Alexandrov space of curvature not smaller than $-1$ with a dense set of sharp points in $U_j \setminus U_{j-1}$. Fix a smooth open set $U_0'$ with $\vm\Subset U_{0}'\Subset U_{1}$, and such that $U_1 \setminus \overline{U_0'}$ is connected. Choose
	\[
	0 < \eps_j \le \frac{1}{1000 \bar C_j^2} \min \Big\{\di_{g_{j-1}}(U_j, \partial U_{j+1}),\di_{g_{j-1}}(U_0',\vm)\Big\}>0 
	\]
in such a way that the tubular neighborhood 
	\[
	V_j \doteq B_{16 \bar C_j\eps_j}^{g_{j-1}}(U_j \setminus U_0') \qquad \text{has smooth boundary}.
	\]
Hereafter the index $j$ will be fixed, so for notational convenience we omit to write it unless it identifies the sets $U_j$. We also use a superscript or subscript $k$ to indicate quantities that refer to the metric $g_{j,k,m_k}$, so for instance we write $|\cdot |_k$, $\vol_k$ to denote the norm and volume, and $B_r^k(z)$ instead of $B_r^{g_{j,k,m_k}}(z)$. Analogously, balls in $M_{j,\infty}$ will be denoted with $B_r^\infty(z)$. By \eqref{bili_2}, we have the following inclusions between tubular neighbourhoods:
	\begin{equation}\label{eq_bonita_bolle}
	B^k_{\eps_j}(U_j \setminus U_0') \Subset V_j \Subset B_{5\eps_j}^k(V_j) \Subset U_{j+1}\setminus \vm \qquad \forall \, k \in \NN.
	\end{equation}
Again using \eqref{bili_2}, we can fix $R_j > 0$ such that	
	\[
	U_{j+1} \Subset B^k_{R_j/2 -1}(o) \qquad \forall \, k \in \NN.
	\] 
Because $\Sec_{g_{j,k,m_k}} \ge -1$ for each $j,k$, on the balls $B^{k}_{R_j}(o)$ we have the validity of  \eqref{eq_doubling}, \eqref{eq_reversedoubling} and \eqref{buser} with constants only depending on $n,p,R_j$. By using \eqref{eq_bonita_bolle}, we can apply Morrey's estimates as stated in \cite[Thm. 9.2.14]{hkst} both to $\varphi_k$ and to $|\nabla \varphi_k|_k$, to deduce that for fixed $j$ there exists a constant $C = C(n,p,R_j)$ such that for each $z \in B_{\eps_j}^k(V_j)$ it holds
	\begin{equation}\label{eq_morrey_0}
	\disp \sup_{x,y \in B^k_{\eps_j}(z)} \frac{|\varphi_k(x) - \varphi_k(y)|}{\di_k(x,y)^{1- \frac{n}{p}}} + \frac{|\nabla \varphi_k(x)|_k - |\nabla \varphi_k(y)|_k}{\di_k(x,y)^{1- \frac{n}{p}}} \le C(n,p,R_j) \eps_j^{\frac{n}{p}} \left( \fint_{B^k_{4\eps_j}(z)} |\nabla \varphi_k|^p_k +  |\nabla^2 \varphi_k|^p_k \right)^\frac{1}{p}
	\end{equation}
Using \eqref{eq_reversedoubling}, \eqref{eq_monotonicity} and \eqref{proprie_contrad}, we get 
	\[
\eps_j^{\frac{n}{p}} \left( \fint_{B^k_{4\eps_j}(z)} |\nabla \varphi_k|^p_k +  |\nabla^2 \varphi_k|^p_k \right)^\frac{1}{p} \le C \eps_j^{\frac{n}{p}} \left( \frac{1}{\eps_j^n \vol_k\big(B^k_{1}(z)\big)} \int_{B^k_{4\eps_j}(z)} |\nabla \varphi_k|^p_k +  |\nabla^2 \varphi_k|^p_k \right)^\frac{1}{p} \le C'.
	\]
Thus \eqref{eq_morrey_0} gives
	\begin{equation}\label{eq_morrey}
	\disp \sup_{x,y \in B^k_{\eps_j}(z)} \frac{|\varphi_k(x) - \varphi_k(y)|}{\di_k(x,y)^{1- \frac{n}{p}}} + \frac{|\nabla \varphi_k(x)|_k - |\nabla \varphi_k(y)|_k}{\di_k(x,y)^{1- \frac{n}{p}}} \le C''(n,p,R_j) \qquad \forall \, z \in B_{\eps_j}^k(V_j).
	\end{equation}
A simple chain argument using \eqref{eq_bilipschitz} then allows to extend the uniform H\"older estimates in \eqref{eq_morrey} to $x,y \in B_{\eps_j}^k(U_j \setminus U_0')$. Briefly, since $V_j$ has smooth boundary we can fix a constant $\hat C_j$ such that, for each $x,y \in V_j$, there exists a curve $\gamma_{xy} \subset V_j$ joining $x$ to $y$ whose length is at most $\hat{C}_j \di_{g_{j-1}}(x,y)$. Restricting to $x,y \in B_{\eps_j}^k(U_j \setminus U_0')$, choose points $\{x_i\}_{i =1}^s$ along $\gamma_{xy}$ in such a way that $x_0 = x$, $x_s = y$ and the length of each subsegment $\gamma_{x_ix_{i+1}}$ with respect to $g_{j-1}$ does not exceed  $\eps_j/(2\hat C_j\bar C_j)$. By \eqref{bili_2}, there exists $\tilde C_j$ such that
	\[
	x_{i} \in B_{\eps_j}^k(x_{i-1}) \quad \forall \, i \in I, \ k \in \NN, \qquad \sum_{i} \di_k(x_i,x_{i+1}) \le \tilde C_j \di_k(x,y).
	\]
Applying \eqref{eq_morrey} with $z = y = x_i$ and $x = x_{i+1}$, and summing up, we get 
	\begin{equation}\label{eq_morrey_stronger}
	|\varphi_k(x) - \varphi_k(y)| + \Big| |\nabla \varphi_k(x)|_k - |\nabla \varphi_k(y)|_k\Big| \le C'''(n,p,R_j) \di_k(x,y)^{1- \frac{n}{p}} \qquad \forall \, x,y \in B_{\eps_j}^k(U_j \setminus U_0').
	\end{equation}

Next, by \eqref{proprie_contrad} and since $M \setminus U_0'$ is connected while $M \setminus \overline{U_j}$ is not, each curve in $M\setminus U_0'$ joining two points $x \in \partial E_{-,j}$, $y \in \partial E_{+,j}$ shall contain a point $x_k \in U_j \setminus U_0'$ for which $\varphi_k(x_k) =0$. 
Hence, $\{\varphi_k\}$ is equibounded on $B_{\eps_j}^k(U_j \setminus U_0')$ and subconverges, by Ascoli-Arzel\'a theorem, pointwise to some $\varphi : B_{\eps_j}^\infty(U_j \setminus U_0') \to \R$ that, because of \eqref{eq_morrey_stronger}, is uniformly continuous on $B_{\eps_j}^\infty(U_j \setminus U_0')$. Furthermore, by \cite[Prop. 3.19]{honda_crelle} and up to subsequences, $\varphi_k \to \varphi$ $L^2$-weakly on each ball $B_{\eps_j}^k(z) \subset B_{\eps_j}^k(U_j \setminus U_0')$, see also \cite[Rem. 3.8]{honda_crelle}.  By H\"older inequality, \eqref{eq_bonita_bolle} and since $(M, g_{j,k,m_k})$ has uniformly bounded volume,
	\[
	\limsup_k \|\varphi_k\|_{W^{1,2}(B_{\eps_j}^k(z), g_{j,k,m_k})} < \infty, 
	\]
and $\varphi_k \to \varphi$ $L^2$-strongly on $B_{\eps_j}^\infty(z)$. By \cite[Thm. 1.3]{honda_crelle}, $\varphi \in W^{1,2}(B_r^\infty(z), \di_\infty)$ for each $r < \eps_j$, $\varphi$ is in the domain of the Laplacian $\mathcal{D}^2(\Delta, B_{\eps_j}^\infty(z))$ on $M_{j,\infty}$ and  
	\begin{equation}\label{eq_convergence}
	\begin{array}{ll}
	\Delta \varphi_k \to \Delta \varphi & \quad \text{$L^2$ weakly on } \, B_{\eps_j}^{\infty}(z) \\[0.3cm]
	\nabla \varphi_k \to \nabla \varphi & \quad \text{$L^2$ strongly on } \, B_r^{\infty}(z), \ \text{ for each } \, r< \eps_j.
	\end{array}
	\end{equation}
In particular, by \cite[Thm. 3.28]{honda_crelle}, $|\nabla \varphi_k| \to |\nabla \varphi|$ $L^2$ strongly on $B_r^\infty(z)$, hence pointwise a.e by \cite[Prop. 3.32]{honda_crelle}. Passing to the limit in \eqref{eq_morrey_stronger}, $\varphi$ and $|\nabla \varphi|$ are uniformly continuous on $\overline{U_j \setminus  U_0}$. If $z$ is a sharp point we apply \cite[Proposition 2.5]{dePNZ} to infer the existence of $\delta_0 = \delta_0(n,z)$ and $\eps'_j = \eps'(n,z,\eps_j) \in (0,\eps_j)$ such that 
	\begin{equation}\label{eq_dephilzim}
	\fint_{B_{r/2}^\infty(z)} |\nabla \varphi|^2 \le (1-\delta_0) \fint_{B_{r}^\infty(z)} |\nabla \varphi|^2
	+ r^2 C(n,z,\eps_j) \fint_{B_r^\infty(z)} (\Delta \varphi)^2 \qquad \forall \, r \le \eps_j'.
	\end{equation}
Using \cite[Thm. 3.29]{honda_crelle} and \eqref{eq_convergence} we deduce that, for every $r \le \eps'_j$,
	\[
	\|\Delta \varphi\|_{L^2(B_r^\infty(z))}  \le \liminf_k \|\Delta \varphi_k\|_{L^2(B_r^k(z))},
	\]  	
hence by H\"older inequality and \eqref{eq_reversedoubling_0} we deduce
	\[
	\begin{array}{lcl}
	\disp r^2 \fint_{B_r^k(z)} |\Delta \varphi_k|^2 & \le & \disp r^2 \vol_k \big(B_r^k(z)\big)^{\frac{-2}{p}} \left(\int_{B_r^k(z)} |\Delta \varphi_k|^p\right)^{\frac{2}{p}} \\[0.5cm]
	& \le & r^2 \vol_k \big(B_r^k(z)\big)^{-\frac{2}{p}} \le (C_\dou')^{-\frac{2}{p}} r^2 \left(\frac{\eps_j}{r}\right)^{\frac{2n}{p}} \vol_k \big(B_{\eps_j}^k(z)\big)^{-\frac{2}{p}} \\[0.3cm]
	& \le & C(n,p,R_j,\eps_j,\nu_j) r^{2 \frac{p-n}{p}}			
	\end{array}
	\]	
where, in the last step, we used again \eqref{eq_monotonicity}. Inserting into \eqref{eq_dephilzim} we eventually obtain
	\[
	\fint_{B_{r/2}^\infty(z)} |\nabla \varphi|^2 \le (1-\delta_0) \fint_{B_r^\infty(z)} |\nabla \varphi|^2 + C(n,p,R_j,\nu_j,z,\eps_j) r^{\frac{2(p-n)}{p}} \qquad \forall \, r \le \eps_j'.
	\]
Consequently, 
	\[
	\lim_{r\to 0} \fint_{B_{r/2}^\infty(z)} |\nabla \varphi|^2 = 0 \qquad \text{for every sharp point $z$.}
	\]
From the uniform continuity of $|\nabla \varphi|$ and the density of the set of sharp points in $U_j \setminus U_{j-1}$, we conclude that $|\nabla \varphi| = 0$ on $\overline{U_j\setminus U_{j-1}}$ (on $U_1 \setminus \vm$, if $j=1$), as claimed. This concludes the proof of Lemma \ref{l: construction}.

\subsection{Proof of Theorem \ref{th_striking}, and Corollaries \ref{cor_gionaludo} and \ref{cor_dePNZ}}\label{striking}

All of them are based on the following simple observation: let $X,Y$ be Riemannian manifolds, with $Y$ compact, and consider a (say, smooth) function $\varphi \in W^{2,p}(X\times Y)$. For every $y \in Y$ fixed, define $\varphi_y : X \to \R$ by $\varphi_y(x) = \varphi(x,y)$. Denote with $\nabla, \bar \nabla$, $\Delta$, $\bar \Delta$, the Levi-Civita connections and the Laplace operator of $X$ and $X \times Y$ respectively. From 
	\[
	|\bar \nabla \varphi(x,y)| \ge |\nabla \varphi_y(x)|, \qquad |\bar \nabla^2 \varphi(x,y)| \ge |\nabla^2 \varphi_y(x)|,
	\]
it holds
	\[
	\|\varphi\|^p_{L^p(X\times Y)} + \|\bar \nabla \varphi\|^p_{L^p(X\times Y)} + \|\bar \nabla^2\varphi\|^p_{L^p(X\times Y)} \ge \int_Y \left\{ \|\varphi_y\|^p_{L^p(X)} + \|\nabla \varphi_y\|^p_{L^p(X)} + \|\nabla^2\varphi_y\|^p_{L^p(X)} \right\} \di y, 
	\]
with equality if $\varphi$ just depends on $y$. Hence, by the definition of $W^{2,p}$ norm, there exists a constant $C_p>0$ only depending on $p$ such that
	\begin{equation}\label{eq_bound_prod}
	\begin{array}{ll}
	\disp \|\varphi\|^p_{W^{2,p}(X\times Y)} \ge C_p \int_Y \|\varphi_y\|^p_{W^{2,p}(X)} \di y & \quad \forall \, \varphi \in C^\infty(X \times Y) \cap W^{2,p}(X\times Y).
	\end{array}
	\end{equation}
Conversely, let $\pi:X\times Y \to X$ be the projection onto the first factor and for any $\psi \in W^{2,p}(X)$ define $\bar \psi \doteq \psi \circ \pi \in W^{2,p}(X\times Y)$. Then 
\[
\begin{array}{ll}
\|\bar \psi\|^p_{L^p(X\times Y)} = \vol(Y)\|\psi\|^p_{L^p(X)}, &\|\bar\nabla\bar \psi\|^p_{L^p(X\times Y)} = \vol(Y)\|\nabla \psi\|^p_{L^p(X)},\\[0.2cm] \|\bar\nabla^2\bar \psi\|^p_{L^p(X\times Y)} = \vol(Y)\|\nabla^2 \psi\|^p_{L^p(X)}, &\|\bar\Delta\bar \psi\|^p_{L^p(X\times Y)} = \vol(Y)\|\Delta \psi\|^p_{L^p(X)}.
\end{array}
\] 
Regarding Theorem\ref{th_striking}, for fixed $n\ge 2$, and $p>2$, consider a surface $M\sharp N$ and the smooth function $F \in W^{k,p}(M\sharp N)$ (for each $k \in \NN$) constructed in Theorem \ref{Counterex} for dimension $2$. In particular, 
	\[
	\|v - F\|_{W^{2,p}(M\sharp N)} \ge 1 \qquad \text{for every } \, v \in C^\infty_c(M\sharp N).
	\]
Consider a compact, boundaryless manifold $Y$ of dimension $n-2$, let $\pi : Q = (M \sharp N) \times Y \to Y$ be the projection onto the second factor, and define $\bar F \doteq F \circ \pi \in W^{2,p}(Q)$. Then, from \eqref{eq_bound_prod}, for every $u \in C^\infty_c(Q)$ it holds 
	\[
	\|u - \bar F\|^p_{W^{k,p}(Q)}	\ge \|u - \bar F\|^p_{W^{2,p}(Q)} \ge C_p \int_Y \|u_y - F\|_{W^{2,p}(M\sharp N)}^p \di y \ge C_p \vol(Y).
	\]
Hence $F\not\in W^{k,p}_0(M)$ and	
	\[
	W_0^{k,p}(Q) \neq W^{k,p}(Q),
	\]
as claimed.\\[0.2cm]
	
As for Corollary \ref{cor_gionaludo}, given $p>2$, let $(M^2,g)$ be a complete surface with $\Sec \ge 0$ constructed in \cite{MV}, so that there exists a sequence $\{F_k\} \subset C^\infty_c(M)$ with $\|F_k\|_{L^p(M)} + \|\Delta F_k\|_{L^p(M)} = 1$ but $\|\nabla^2 F_k\|_{L^p(M)} \to \infty$. Fix a compact manifold $Y^{n-2}$ with $\Sec \ge 0$, and define as above $\bar F_k=F_k\circ \pi\in C^\infty_c(M\times X)$. It is immediate to deduce that
\[
\|\bar F_k\|_{L^p(M\times X)} + \|\bar \Delta \bar F_k\|_{L^p(M\times X)} = \vol(Y)^{1/p},\qquad \text{but}\qquad \|\bar\nabla^2 \bar F_k\|_{L^p(M)} \to \infty.
\]\\[0.2cm]
Corollary \ref{cor_dePNZ} can be proved in a very similar way, starting from a sequence of compact $2$-dimensional positively curved manifolds $M_k$ and a sequence of functions $F_k\in C^\infty(M_k)$ which verify
\[
\|F_k\|_{L^p(M_k)} + \|\Delta F_k\|_{L^p(M_k)} = \vol(Y)^{- \frac{1}{p}},\qquad\text{but}\qquad \|\nabla^2 F_k\|_{L^p(M_k)} \to \infty;
\]
the existence of these sequences is guaranteed by \cite{dePNZ}.

\begin{acknowledgement*}
M.R. and G.V. are members of INdAM-GNAMPA. We acknowledge that the present research has been partially supported by PRIN project 2017 ``Real and Complex Manifolds: Topology, Geometry and holomorphic dynamics'' and by MIUR grant Dipartimenti di Eccellenza 2018-2022 (E11G18000350001), DISMA, Politecnico di Torino. S.H. acknowledges the support of the Grant-in-Aid for Scientific Research (B) of 20H01799 and the Grant-in-Aid for Scientific Research (B) of 18H01118. We also would like to thank Li Chen and Willie WY Wong for pointing out some literature and the anonymous referee for the careful reading of the paper and for useful suggestions.
\end{acknowledgement*}

\bibliographystyle{amsplain}
\bibliography{CutOffs}

\providecommand{\bysame}{\leavevmode\hbox to3em{\hrulefill}\thinspace}
\providecommand{\MR}{\relax\ifhmode\unskip\space\fi MR }
\providecommand{\MRhref}[2]{%
  \href{http://www.ams.org/mathscinet-getitem?mr=#1}{#2}
}
\providecommand{\href}[2]{#2}
\begin{thebibliography}{10}

\bibitem{ambrosio_survey}
Luigi Ambrosio, \emph{Calculus, heat flow and curvature-dimension bounds in
  metric measure spaces}, Proceedings of the {I}nternational {C}ongress of
  {M}athematicians---{R}io de {J}aneiro 2018. {V}ol. {I}. {P}lenary lectures,
  World Sci. Publ., Hackensack, NJ, 2018, pp.~301--340. \MR{3966731}

\bibitem{aubin-bull}
Thierry Aubin, \emph{Espaces de {S}obolev sur les vari\'{e}t\'{e}s
  {r}iemanniennes}, Bull. Sci. Math. (2) \textbf{100} (1976), no.~2, 149--173.
  \MR{0488125}

\bibitem{bakry}
Dominique Bakry, \emph{\'{E}tude des transformations de {R}iesz dans les
  vari\'{e}t\'{e}s riemanniennes \`a courbure de {R}icci minor\'{e}e},
  S\'{e}minaire de {P}robabilit\'{e}s, {XXI}, Lecture Notes in Math., vol.
  1247, Springer, Berlin, 1987, pp.~137--172. \MR{941980}

\bibitem{Bandara}
Lashi Bandara, \emph{Density problems on vector bundles and manifolds}, Proc.
  Amer. Math. Soc. \textbf{142} (2014), no.~8, 2683--2695. \MR{3209324}

\bibitem{BianchiSetti}
Davide Bianchi and Alberto~G. Setti, \emph{Laplacian cut-offs, porous and fast
  diffusion on manifolds and other applications}, Calc. Var. Partial
  Differential Equations \textbf{57} (2018), no.~1, Art. 4, 33. \MR{3735744}

\bibitem{BH}
Martin~R. Bridson and Andr\'{e} Haefliger, \emph{Metric spaces of non-positive
  curvature}, Grundlehren der Mathematischen Wissenschaften [Fundamental
  Principles of Mathematical Sciences], vol. 319, Springer-Verlag, Berlin,
  1999. \MR{1744486}

\bibitem{buyalo}
S.~V. Bujalo, \emph{Shortest paths on convex hypersurfaces of a {R}iemannian
  space}, Zap. Nau\v{c}n. Sem. Leningrad. Otdel. Mat. Inst. Steklov. (LOMI)
  \textbf{66} (1976), 114--132, 207, Studies in topology, II. \MR{0643664}

\bibitem{buser}
Peter Buser, \emph{A note on the isoperimetric constant}, Ann. Sci. \'{E}cole
  Norm. Sup. (4) \textbf{15} (1982), no.~2, 213--230. \MR{683635}

\bibitem{Car}
Gilles Carron, \emph{Riesz transform on manifolds with quadratic curvature
  decay}, Rev. Mat. Iberoam. \textbf{33} (2017), no.~3, 749--788. \MR{3713030}

\bibitem{CarCouHas}
Gilles Carron, Thierry Coulhon, and Andrew Hassell, \emph{Riesz transform and
  {$L^p$}-cohomology for manifolds with {E}uclidean ends}, Duke Math. J.
  \textbf{133} (2006), no.~1, 59--93. \MR{2219270}

\bibitem{LiChen}
Li~Chen, \emph{Sub-{G}aussian heat kernel estimates and quasi {R}iesz
  transforms for {$1\leq p\leq 2$}}, Publ. Mat. \textbf{59} (2015), no.~2,
  313--338. \MR{3374610}

\bibitem{CCFR}
Li~Chen, Thierry Coulhon, Joseph Feneuil, and Emmanuel Russ, \emph{Riesz
  transform for {$1\le p\le 2$} without {G}aussian heat kernel bound}, J. Geom.
  Anal. \textbf{27} (2017), no.~2, 1489--1514. \MR{3625161}

\bibitem{ChengThalmaierThompson}
Li-Juan Cheng, Anton Thalmaier, and James Thompson, \emph{Quantitative
  {$C^1$}-estimates by {B}ismut formulae}, J. Math. Anal. Appl. \textbf{465}
  (2018), no.~2, 803--813. \MR{3809330}

\bibitem{CD_TAMS}
Thierry Coulhon and Xuan~Thinh Duong, \emph{Riesz transforms for {$1\leq p\leq
  2$}}, Trans. Amer. Math. Soc. \textbf{351} (1999), no.~3, 1151--1169.
  \MR{1458299}

\bibitem{CD}
\bysame, \emph{Riesz transform and related inequalities on noncompact
  {R}iemannian manifolds}, Comm. Pure Appl. Math. \textbf{56} (2003), no.~12,
  1728--1751. \MR{2001444}

\bibitem{dePNZ}
Guido De~Philippis and Jes\'{u}s N\' {u}\~{n}ez Zimbr\'{o}n, \emph{The behavior
  of harmonic functions at singular points of $\mathsf{RCD}$ spaces}, ArXiv
  Preprint Server -- arXiv:1909.05220, 2019.

\bibitem{Eichhorn2}
J{\"u}rgen Eichhorn, \emph{Elliptic differential operators on noncompact
  manifolds}, Seminar {A}nalysis of the {K}arl-{W}eierstrass-{I}nstitute of
  {M}athematics, 1986/87 ({B}erlin, 1986/87), Teubner-Texte Math., vol. 106,
  Teubner, Leipzig, 1988, pp.~4--169.

\bibitem{Eichhorn1}
J\"{u}rgen Eichhorn, \emph{Global analysis on open manifolds}, Nova Science
  Publishers, Inc., New York, 2007. \MR{2343536}

\bibitem{Friedman}
Avner Friedman, \emph{Partial differential equations}, Holt, Rinehart and
  Winston, Inc., New York-Montreal, Que.-London, 1969. \MR{0445088}

\bibitem{ghomi}
Mohammad Ghomi, \emph{The problem of optimal smoothing for convex functions},
  Proc. Amer. Math. Soc. \textbf{130} (2002), no.~8, 2255--2259. \MR{1896406}

\bibitem{GGP-AccFinn}
Davide Guidetti, Batu G\"{u}neysu, and Diego Pallara, \emph{{$L^1$}-elliptic
  regularity and {$H=W$} on the whole {$L^p$}-scale on arbitrary manifolds},
  Ann. Acad. Sci. Fenn. Math. \textbf{42} (2017), no.~1, 497--521. \MR{3558546}

\bibitem{Guneysu-Book}
Batu G{\"u}neysu, \emph{Covariant {S}chr{\"o}dinger semigroups on {R}iemannian
  manifolds}, Operator Theory: Advances and Applications, vol. 264,
  Birkh{\"a}user/Springer, Cham, 2017. \MR{3751359}

\bibitem{GuneysuPigola}
Batu G{\"u}neysu and Stefano Pigola, \emph{The {C}alder{\'o}n-{Z}ygmund
  inequality and {S}obolev spaces on noncompact {R}iemannian manifolds}, Adv.
  Math. \textbf{281} (2015), 353--393. \MR{3366843}

\bibitem{GuneysuPigola_AMPA}
Batu G\"{u}neysu and Stefano Pigola, \emph{{$L^p$}-interpolation inequalities
  and global {S}obolev regularity results}, Ann. Mat. Pura Appl. (4)
  \textbf{198} (2019), no.~1, 83--96, With an appendix by Ognjen Milatovic.
  \MR{3918620}

\bibitem{Hebey}
Emmanuel Hebey, \emph{Sobolev spaces on {R}iemannian manifolds}, Lecture Notes
  in Mathematics, vol. 1635, Springer-Verlag, Berlin, 1996. \MR{1481970}

\bibitem{HebeyCourant}
\bysame, \emph{Nonlinear analysis on manifolds: {S}obolev spaces and
  inequalities}, Courant Lecture Notes in Mathematics, vol.~5, New York
  University, Courant Institute of Mathematical Sciences, New York; American
  Mathematical Society, Providence, RI, 1999. \MR{1688256}

\bibitem{hkst}
Juha Heinonen, Pekka Koskela, Nageswari Shanmugalingam, and Jeremy~T. Tyson,
  \emph{Sobolev spaces on metric measure spaces}, New Mathematical Monographs,
  vol.~27, Cambridge University Press, Cambridge, 2015, An approach based on
  upper gradients. \MR{3363168}

\bibitem{honda_crelle}
Shouhei Honda, \emph{Ricci curvature and {$L^p$}-convergence}, J. Reine Angew.
  Math. \textbf{705} (2015), 85--154. \MR{3377391}

\bibitem{IRV-HO}
Debora Impera, Michele Rimoldi, and Giona Veronelli, \emph{Higher order
  distance-like functions and {S}obolev spaces}, ArXiv Preprint Server --
  arXiv:1908.10951, 2019.

\bibitem{IRV-HessCutOff}
\bysame, \emph{Density problems for second order {S}obolev spaces and cut-off
  functions on manifolds with unbounded geometry}, Int. Math. Res. Not. IMRN.
  Online first. (DOI: 10.1093/imrn/rnz131).

\bibitem{koreschoen}
Nicholas~J. Korevaar and Richard~M. Schoen, \emph{Global existence theorems for
  harmonic maps to non-locally compact spaces}, Comm. Anal. Geom. \textbf{5}
  (1997), no.~2, 333--387. \MR{1483983}

\bibitem{Li}
Siran Li, \emph{Counterexamples to the {$L^p$}-{C}alder\'{o}n-{Z}ygmund
  estimate on open manifolds}, Ann. Global Anal. Geom. \textbf{57} (2020),
  no.~1, 61--70. \MR{4057451}

\bibitem{MV}
Ludovico Marini and Giona Veronelli, \emph{The {$L^p$} {C}alder\'{o}n-{Z}ygmund
  inequality on non-compact manifolds of positive curvature}, to appear in Ann.
  Global Anal. Geom. ArXiv Preprint Server -- arXiv:2011.13025, 2020.

\bibitem{OtsuShioya}
Yukio Otsu and Takashi Shioya, \emph{The {R}iemannian structure of {A}lexandrov
  spaces}, J. Differential Geom. \textbf{39} (1994), no.~3, 629--658.
  \MR{1274133}

\bibitem{petersen}
Peter Petersen, \emph{Riemannian geometry}, third ed., Graduate Texts in
  Mathematics, vol. 171, Springer, Cham, 2016. \MR{3469435}

\bibitem{Pigola_ArXiv}
Stefano Pigola, \emph{Global {C}alder\'on-{Z}ygmund inequalities on complete
  {R}iemannian manifolds}, ArXiv Preprint Server -- arXiv:2011.03220, 2020.

\bibitem{saloff}
Laurent Saloff-Coste, \emph{Aspects of {S}obolev-type inequalities}, London
  Mathematical Society Lecture Note Series, vol. 289, Cambridge University
  Press, Cambridge, 2002. \MR{1872526}

\bibitem{Sampson}
J.~H. Sampson, \emph{On a theorem of {C}hern}, Trans. Amer. Math. Soc.
  \textbf{177} (1973), 141--153. \MR{0317221}

\bibitem{V-counter}
Giona Veronelli, \emph{Sobolev functions without compactly supported
  approximations}, to appear in Anal. PDE. ArXiv Preprint Server --
  arXiv:2004.10682, 2020.

\bibitem{Yau}
Shing~Tung Yau, \emph{Some function-theoretic properties of complete
  {R}iemannian manifold and their applications to geometry}, Indiana Univ.
  Math. J. \textbf{25} (1976), no.~7, 659--670. \MR{417452}

\end{thebibliography}
\end{document}